\documentclass[11pt]{amsart}
\usepackage[leqno]{amsmath}
\usepackage{amssymb}

\usepackage{enumerate}
\usepackage[dvipsnames]{xcolor}
\usepackage{comment}
\usepackage{subfigure}
\usepackage{microtype}
\usepackage{mathrsfs}
\usepackage{graphicx}
\usepackage[all]{xy}
\numberwithin{equation}{section}
\numberwithin{figure}{section}

\definecolor{ceylon}{RGB}{20,190,170}
\definecolor{darkblue}{RGB}{0,50,190}
\definecolor{darkred}{RGB}{190,0,0}

\usepackage{hyperref}
\hypersetup{
    colorlinks=true,
    linkcolor=MidnightBlue,
    citecolor=BrickRed,
    urlcolor=blue,
    pdfborder={0 0 0}
}

\makeatletter
\newcommand{\labitem}[2]{
\def\@itemlabel{\textbf{#1}}
\item
\def\@currentlabel{#1}\label{#2}}
\makeatother

\newtheorem{theorem}{Theorem}[section]
\newtheorem{remark}[theorem]{Remark}
\newtheorem{lemma}[theorem]{Lemma}
\newtheorem{proposition}[theorem]{Proposition}

\newtheorem{assumption}[theorem]{Assumption}

\newcommand{\eps}{\epsilon}

\newcommand{\C}{\mathbf{C}}
\newcommand{\D}{\mathbf{D}}
\newcommand{\E}{\mathbf{E}}

\newcommand{\h}{\mathbf{H}}
\renewcommand{\H}{\mathbf{H}}
\newcommand{\N}{\mathbf{N}}
\newcommand{\Z}{\mathbf{Z}}
\newcommand{\p}{\mathbf{P}}

\newcommand{\R}{\mathbf{R}}

\newcommand{\Fh}{\mathfrak {h}}
\newcommand{\Fp}{\mathfrak {p}}

\newcommand{\CC}{\mathcal {C}}

\newcommand{\CF}{\mathcal {F}}

\newcommand{\CN}{\mathcal {N}}

\newcommand{\CW}{\mathcal {W}}

\newcommand{\CZ}{\mathcal {Z}}

\newcommand{\SLE}{{\rm SLE}}

\newcommand{\dist}{\mathrm{dist}}

\newcommand{\diam}{\mathrm{diam}}

\newcommand{\im}{\mathrm{Im}}
\newcommand{\re}{\mathrm{Re}}

\newcommand{\one}{{\bf 1}}

\newcommand{\wt}{\widetilde}
\newcommand{\ol}{\overline}
\newcommand{\ul}{\underline}

\newcommand{\giv}{\,|\,}

\newcommand{\qdist}{\mathfrak d}

\newcommand{\cyl}{\mathscr{C}}

\newcommand{\wh}{\widehat}

\begin{document}

\title[The geodesics in LQG are not $\SLE$s]{The geodesics in Liouville quantum gravity are not Schramm-Loewner evolutions}

\author{Jason Miller}

\address{Statistical Laboratory, Center for Mathematical Sciences, University of Cambridge, Wilberforce Road, Cambridge CB3 0WB, UK}
\email {jpmiller@statslab.cam.ac.uk}

\author{Wei Qian}

\address{Statistical Laboratory, Center for Mathematical Sciences, University of Cambridge, Wilberforce Road, Cambridge CB3 0WB, UK}
\email{wq214@cam.ac.uk}

\begin{abstract}
We prove that the geodesics associated with any metric generated from Liouville quantum gravity (LQG) which satisfies certain natural hypotheses are necessarily singular with respect to the law of any type of $\SLE_\kappa$.  These hypotheses are satisfied by the LQG metric for $\gamma=\sqrt{8/3}$ constructed by the first author and Sheffield, and subsequent work by Gwynne and the first author has shown that there is a unique metric which satisfies these hypotheses for each $\gamma \in (0,2)$.  As a consequence of our analysis, we also establish certain regularity properties of LQG geodesics which imply, among other things, that they are conformally removable.
\end{abstract}
\date{\today}
\maketitle

\section{Introduction}

Suppose that $D \subseteq \C$ is a domain and $h$ is an instance of the Gaussian free field (GFF) $h$ on~$D$.  Fix $\gamma \in (0,2]$.  The $\gamma$-Liouville quantum gravity (LQG) surface described by $h$ is the random Riemannian manifold with metric tensor
\begin{equation}
\label{eqn:lqg_metric}
e^{\gamma h(z)} (dx^2 + dy^2)
\end{equation}
where $dx^2 + dy^2$ denotes the Euclidean metric tensor.  This expression is ill-defined as $h$ is a distribution and not a function, hence does not take values at points.  The volume form associated with~\eqref{eqn:lqg_metric} was constructed by Duplantier-Sheffield in \cite{ds2011kpz} (though measures of this type were constructed earlier by Kahane \cite{k1985gmc} under the name Gaussian multiplicative chaos; see also \cite{hk1971quantum}).  The construction in the case $\gamma \in (0,2)$ proceeds by letting for each $z \in D$ and $\epsilon > 0$ with $B(z,\epsilon) \subseteq D$, $h_\epsilon(z)$ be the average of $h$ on $\partial B(z,\epsilon)$ and then taking
\begin{equation}
\label{eqn:lqg_measure_def}
 \mu_h^\gamma = \lim_{\epsilon \to 0} \epsilon^{\gamma^2/2} e^{\gamma h_\epsilon(z)} dz
\end{equation}
where $dz$ denotes Lebesgue measure on $D$.  The construction in the case $\gamma = 2$ is similar but with the normalization factor taken to be $\sqrt{\log \epsilon^{-1}} \epsilon^2$ \cite{drsv2014critical_deriv,drsv2014critical_kpz}.  The limiting procedure~\eqref{eqn:lqg_measure_def} implies that the measures $\mu_h^\gamma$ satisfy a certain change of coordinates formula.  In particular, suppose that $h$ is a GFF on $D$, $\varphi \colon \wt{D} \to D$ is a conformal transformation, and
\begin{equation}
\label{eqn:coord_change}
\wt{h} = h \circ \varphi + Q \log|\varphi'| \quad\mathrm{where}\quad Q = \frac{2}{\gamma} + \frac{\gamma}{2}.
\end{equation}
If $\mu_{\wt{h}}^\gamma$ is the $\gamma$-LQG measure associated with $\wt{h}$, then we have that $\mu_h^\gamma(\varphi(A)) = \mu_{\wt{h}}^{\gamma}(A)$ for all Borel sets $A \subseteq \wt{D}$.  The relation~\eqref{eqn:coord_change} is referred to as the \emph{coordinate change formula} in LQG.  Two domain/field pairs $(D,h)$, $(\wt{D},\wt{h})$ are said to be \emph{equivalent as quantum surfaces} if $h$, $\wt{h}$ are related as in~\eqref{eqn:coord_change}.  A \emph{quantum surface} is an equivalence class with respect to this equivalence relation and a representative is referred to as an \emph{embedding} of a quantum surface.

The purpose of the present work is to study the properties of geodesics for $\gamma$-LQG surfaces and their relationship with the Schramm-Loewner evolution ($\SLE$) \cite{s2000scaling}.  Since the GFF is conformally invariant and satisfies the spatial Markov property, one is led to wonder whether the geodesics in $\gamma$-LQG should satisfy Schramm's  conformal Markov characterization of $\SLE$ (see Section~\ref{subsec:sle} for a review) and hence be given by $\SLE$-type curves (see \cite[Problem 3, Section 5]{b2010icm} as well as Sections~2 and~4 from the open problems from \cite{diabconf}).  
As pointed out by Duplantier \cite[Section 4]{diabconf}, evidence in support of the relationship between $\SLE$ and LQG geodesics is given by the fact that the exponent for having $k$ geodesics connect a pair of points in a random planar map \cite{bg2008geodesics} matches the exponent for having $k$ disjoint self-avoiding walks on a random planar map connect a pair of points \cite{dk1990,dk1998spectra} (and self-avoiding walks on random quadrangulations were proven to converge to SLE$_{8/3}$ \cite{gm2016saw}).
Since a geodesic is necessarily a simple curve, it can possibly be an  $\SLE_\kappa$ curve only for $\kappa \in (0,4]$,  as $\SLE_\kappa$ curves with $\kappa > 4$ are self-intersecting \cite{rs2005basic}.  The main result of the present work is to show that the geodesics in $\gamma$-LQG are in fact singular with respect the law of any type of $\SLE_\kappa$.

Prior to this work, the metric space structure for LQG had only been constructed for $\gamma=\sqrt{8/3}$ in \cite{ms2015lqg_tbm1, ms2016lqg_tbm2,ms2016lqg_tbm3,ms2015axiomatic}.  In this case, the resulting metric measure space is equivalent to that of a \emph{Brownian surface}, the Gromov-Hausdorff scaling limit of uniformly random planar maps.  The first result of this type was proved by Le Gall \cite{lg2013uniqueness} and Miermont \cite{m2013brownian} for uniformly random quadrangulations of the sphere.  The works \cite{lg2013uniqueness,m2013brownian} have since been extended to the case of uniformly random quadrangulations of the whole-plane \cite{cl2014bp}, the disk \cite{bm2017disk,gm2017disk}, and the half-plane \cite{bmr2016classification,gm2017halfplane}.  The type of Brownian surface that one obtains from the $\sqrt{8/3}$-LQG metric depends on the type of GFF $h$.  Following this work, the metric for $\gamma$-LQG was constructed for $\gamma \in (0,2)$ in \cite{gm2019exunique,gm2019conf}, building on \cite{gm2019localmetrics,gm2019confluence,dddf2019tightness} and some ideas from the present work.  The results of this article in particular apply to the LQG metric for all $\gamma \in (0,2)$ but we emphasize that this work is independent of \cite{ms2015lqg_tbm1, ms2016lqg_tbm2,ms2016lqg_tbm3,ms2015axiomatic} and precedes \cite{gm2019exunique,gm2019conf}.

\newcommand{\qball}{B}

 We will first look at a metric $\qdist_h$ in $\C$ associated with a whole-plane GFF instance $h$ which satisfies the following assumption. 
We let $\qball_h(z,r)$ denote the open metric ball under $\qdist_h$ centered at $z$ with radius $r > 0$.
\begin{assumption}
\label{asump:main}
We assume that $\qdist_h$ is an $h$-measurable metric which is homeomorphic to the Euclidean metric on $\C$ and which satisfies:
\begin{enumerate}[(i)]
\item\label{it:locality} Locality: for all $z \in \C$ and $r > 0$, $\ol{B_h(z,r)}$ is a local set for $h$.
\item\label{it:scaling} Scaling: there exists a constant $\beta > 0$ such that for each $C \in \R$ we have that $\qdist_{h + C}(x,y) = e^{\beta C} \qdist_h(x,y)$.

\item\label{it:coord_change} Compatibility with affine maps: if $\varphi \colon \C \to \C$ is an affine map (combination of scaling and translation) and $\wt{h} = h \circ \varphi + Q\log|\varphi'|$ then $\qdist_{\wt{h}}(z,w) = \qdist_h(\varphi(z),\varphi(w))$ for all $z,w \in \C$.
\end{enumerate}
\end{assumption}
\noindent (We will review the definition of GFF local sets in Section~\ref{subsec:gff}.)  Since the whole-plane GFF is only defined modulo an additive constant, to be concrete we will often fix the additive constant by taking the average of the field on $\partial \D$ to be equal to~$0$.  
Recall that a metric space $(X,d)$ is said to be \emph{geodesic} if for every $x,y \in X$ there exists a path in $X$ connecting $x$ to $y$ with length equal to $d(x,y)$.  
The metric space $(\C,\qdist_h)$ is geodesic, due to the Hopf-Rinow theorem and the fact that it is complete and locally compact being homeomorphic to the Euclidean whole-plane.  We emphasize that the geodesics of $\qdist_h$ are the same as those of $\qdist_{h+C}$ by part~\eqref{it:scaling} of Assumption~\ref{asump:main}, so the particular manner in which we have fixed the additive constant is not important for the purpose of analyzing the properties of geodesics. 

Note that Assumption~\ref{asump:main} was shown to hold in the case $\gamma=\sqrt{8/3}$ in \cite{ms2015lqg_tbm1,ms2016lqg_tbm2,ms2016lqg_tbm3}. 
After the present article, it was established in \cite{gm2019exunique,gm2019conf} that there exists a unique  metric satisfying (an equivalent form of) Assumption~\ref{asump:main}  for each $\gamma \in (0,2)$. We also expect there exists a unique metric satisfying Assumption~\ref{asump:main} for the case $\gamma=2$, though this is not proved in \cite{gm2019exunique,gm2019conf}.

Given a metric space $(\C,\qdist_h)$ satisfying Assumption~\ref{asump:main}, for any general domain $D\subseteq \C$, we can define a metric space $(D, \qdist_{h,D})$ where $\qdist_{h,D}$ is the \emph{internal metric} on $D$ induced by $\qdist_h$, i.e., for any $x,y\in D$, $\qdist_{h,D} (x,y):=\inf_\eta \ell (\eta)$ where the infimum is taken over all $\qdist_h$-rectifiable curves $\eta$ connecting $x$ and $y$ that are contained in $D$ and $\ell(\eta)$ is the $\qdist_h$-length of $\eta$. 
Recall that $(X,d)$ is said to be a \emph{length space} if for every $\epsilon > 0$ and $x,y \in X$ there exists a path $\eta$ connecting $x$ and $y$ with length at most $d(x,y) + \epsilon$. By definition, $(D, \qdist_{h,D})$ is a length space.
Note that the metric $\qdist_{h,D}$ is entirely determined by $\ol {B_h(z,r)}$ for all $z\in D$ and $r\in(0, \qdist_h(z, \partial D))$, hence part~\eqref{it:locality} of Assumption~\ref{asump:main} implies that $\qdist_{h,D}$ is measurable with respect to the restriction of $h$ on $D$, denoted by $h|_D$.
Finally, we can also consider a GFF $\wt h$ on $D$ with more general boundary conditions, for example piecewise constant or free. For any domain $U\subseteq D$ with positive distance to $\partial D$, the law of $\wt h|_U$ is equal to $h|_U$ plus a (possibly random) continuous function in $U$, hence one can define $(D, \qdist_{\wt h})$ from $(D, \qdist_{h,D})$ by part~\eqref{it:scaling} of Assumption~\ref{asump:main}.

In the present article, we will work with $D=\C$ and a whole-plane GFF $h$.
However, the a.s.\ properties that we will establish for geodesics in this work for the whole-plane GFF transfer to the setting of the GFF on a general domain $D \subseteq \C$ (or to the other types of quantum surfaces considered in \cite{dms2014mating}) by absolute continuity.
To explain this point further, suppose that $D \subseteq \C$ and $\wt{h}$ is a GFF on $D$.  Suppose that $U \subseteq D$ is open, bounded and has positive distance from $\partial D$.  We fix the additive constant for $h$ so that its average on a circle which is disjoint from $U$ is equal to~$0$.  (As we mentioned above, the particular manner in which we fix the additive constant is for technical convenience and does not change the a.s.\ properties of the geodesics.)  With the additive constant for $h$ fixed in this way, the law of $\wt{h}|_U$ is mutually absolutely continuous with respect to the law of $h|_U$.  Consequently, if $x,y \in U$ then on the event that $\qdist_{\wt{h}}(x,y)$ is less than the $\qdist_{\wt{h}}$-distance from $x$ to $\partial U$ we have (Theorem~\ref{thm:geo_unique}) that there is a.s.\ a unique $\qdist_{\wt{h}}$-geodesic connecting $x$ and $y$ whose law is absolutely continuous with respect to the law of the a.s.\ unique $\qdist_h$ geodesic connecting $x$ and $y$.  As all of our other theorems are a.s.\ results, they thus apply to this geodesic on this event.  On the event that $\qdist_{\wt{h}}(x,y)$ is larger than the $\qdist_{\wt{h}}$-distance of $x$ to $\partial U$, it is possible that a $\qdist_{\wt{h}}$-geodesic from $x$ to $y$ can hit $\partial U$.  
However, the proofs of our main results in fact apply to \emph{all} geodesics simultaneously for the whole-plane case and so similar absolute continuity type arguments allow us to make statements about $\qdist_{\wt{h}}$-geodesics whenever they are away from the domain boundary.

\begin{figure}[ht!]
\begin{center}
\includegraphics[width=0.49\textwidth]{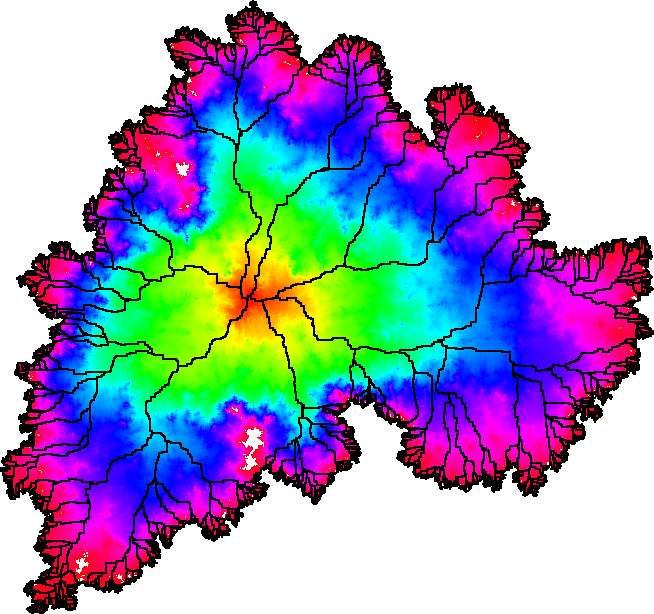} \hspace{0.00\textwidth}	
\includegraphics[width=0.49\textwidth]{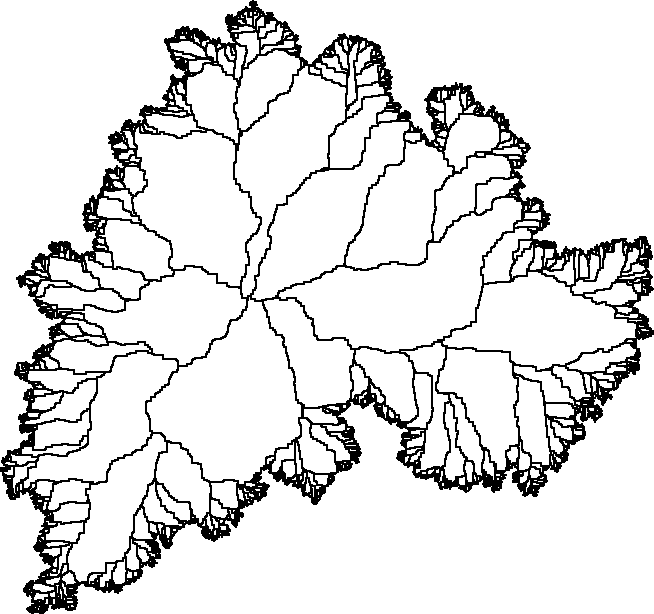}	
\end{center}
\caption{\label{fig:metric_ball_simulation} {\bf Left:} A metric ball constructed using a discretization of $\sqrt{8/3}$-LQG together with all of the geodesics from the outer boundary of the ball to its center.  The different colors indicate the distance of the points to the center.  {\bf Right:} Only the geodesics from the outer boundary of the ball to its center are shown.}
\end{figure}

Our first main result is the a.s.\ uniqueness of geodesics connecting generic points in our domain.

\begin{theorem}
\label{thm:geo_unique}
Suppose that $h$ is a whole-plane GFF with the additive constant fixed as above and that $x,y \in \C$ are distinct.  There is a.s.\ a unique $\qdist_h$-geodesic $\eta$ connecting $x$ and $y$.
\end{theorem}

We note that Theorem~\ref{thm:geo_unique} was shown to hold for $\gamma=\sqrt{8/3}$ in \cite{ms2015lqg_tbm1, ms2016lqg_tbm2,ms2016lqg_tbm3} when $x$, $y$ are taken to be quantum typical (i.e., ``sampled'' from $\mu_h$).  Taking $x$, $y$ to be quantum typical corresponds to adding $-\gamma\log|\cdot|$ singularities at deterministic points $x$, $y$ (see, e.g., \cite{ds2011kpz}).  The proof of Theorem~\ref{thm:geo_unique} given in the present work applies to this setting for $\gamma=\sqrt{8/3}$, but is also applicable in greater generality.  Theorem~\ref{thm:geo_unique} will be important because it allows us to refer to \emph{the} geodesic connecting generic points $x,y$.  We emphasize that Theorem~\ref{thm:geo_unique} does not rule out the existence of exceptional points between which there are multiple geodesics, which are known to exist in the case $\gamma=\sqrt{8/3}$.

Our next main result answers the question mentioned above about the relationship between LQG geodesics and $\SLE$.  Recall that whole-plane $\SLE$ is the variant which describes a random curve connecting two points in the Riemann sphere, so it is the natural one to compare to LQG geodesics.

\begin{theorem}
\label{thm:singular}
Suppose that $h$ is a whole-plane GFF with the additive constant fixed as above and that $x,y \in \C$ are distinct.  Let $\eta$ be the a.s.\ unique geodesic from $x$ to $y$.  The law of $\eta$ is singular with respect to the law of a whole-plane $\SLE_\kappa$ curve from $x$ to $y$ for any value of $\kappa >0$.
\end{theorem}

As mentioned above, the proof of Theorem~\ref{thm:singular} applies in other settings as well.  For example, the same technique applies to show in the case where $D \subseteq \C$ is a simply connected domain that the law of a geodesic between distinct boundary points (resp.\ a boundary point to an interior point) is singular with respect to chordal (resp.\ radial) $\SLE$.

We will prove Theorem~\ref{thm:singular} by analyzing the fine geometric properties of geodesics in LQG.  In particular, we will show that geodesics in LQG are in a certain sense much more regular than $\SLE$ curves.  As a consequence of our analysis, we will obtain the following theorem which serves to quantify this regularity (in a reparameterization invariant manner).

\begin{theorem}
\label{thm:regularity}
Suppose that $h$ is a whole-plane GFF with the additive constant fixed as above. Almost surely, the following is true. For any $\qdist_h$-geodesic $\eta$ (between any two points in $\C$)  and for any parameterization of $\eta$ with time interval $[0,T]$, for each $\delta \in (0,1)$ there exists a constant $C(\delta,\eta) > 0$ so that
\begin{equation}
\label{eqn:regularity}
\diam(\eta([s,t])) \leq C(\delta,\eta) |\eta(t)-\eta(s)|^{1-\delta} \quad\text{for all}\quad s,t \in [0,T].
\end{equation}
\end{theorem}
Let us point out that the regularity condition in the theorem above a.s.\ holds for all $\qdist_h$-geodesics simultaneously, even though the statement in Theorem~\ref{thm:singular} holds a.s.\ only for fixed $x, y\in \C$ (since in the latter setting, one has to choose a geodesic, before comparing its law with SLE).

An important concept in the theory of LQG is \emph{conformal removability}.  Recall that a compact set $K \subseteq \C$ is said to be \emph{conformally removable} if the following is true.  Suppose that $U \supseteq K$ is an open set and $\varphi \colon U \to V$ is a homeomorphism which is conformal on $U \setminus K$.  Then $\varphi$ is conformal on $U$.

\begin{theorem}
\label{thm:removable}
Suppose that $h$ is a whole-plane GFF with the additive constant fixed as above.  Then almost surely, any $\qdist_h$-geodesic is conformally removable.
\end{theorem}

The conformal removability of a path in LQG is important because it implies that a conformal welding in which the path arises as the gluing interface is uniquely determined (see, e.g., \cite{s2016zipper,dms2014mating,mmq2018welding}).  In the case $\gamma=\sqrt{8/3}$, the conformal removability of geodesics is especially important as it is shown in \cite{ms2015axiomatic} that metric balls in the Brownian map can be decomposed into independent slices obtained by cutting the metric ball along the geodesics from its outer boundary to its center (see Figure~\ref{fig:metric_ball_simulation}).  Theorem~\ref{thm:removable} implies in the context of $\sqrt{8/3}$-LQG that the conformal structure associated with a metric ball is uniquely determined by these slices.  We also note that the conformal removability of geodesics in the case $\gamma=\sqrt{8/3}$ was posed as \cite[Problem~9.3]{ms2016lqg_tbm2} and Theorem~\ref{thm:removable} solves this problem.  We will prove Theorem~\ref{thm:removable} by checking that a sufficient condition for conformal removability due to Jones-Smirnov \cite{js2000remove} is necessarily satisfied for the geodesics in LQG using Theorem~\ref{thm:regularity} and an a.s.\ upper bound on the upper Minkowski dimension for the geodesics in LQG (see Proposition~\ref{prop:dimension_bound}) which is strictly smaller than $2$.

We finish by mentioning that there are many other sets of interest that one can generate using a metric from LQG.  Examples include the boundaries of metric balls (see Figure~\ref{fig:metric_ball_simulation}) and the boundaries of the cells formed in a Poisson-Voronoi tessellation (see \cite{gms2018poissonvoronoi}).  We expect that the techniques developed in this paper could be used to show that these sets are both not given by any form of $\SLE$ curve and also are conformally removable.  This leaves one to wonder whether there is any natural set that one can generate from a metric for LQG which is an $\SLE$.

\subsection*{Outline}

The remainder of this article is structured as follows.  We begin in Section~\ref{sec:preliminaries} by reviewing a few of the basic facts about the GFF and $\SLE$ which will be important for this work.  Next, in Section~\ref{sec:uniqueness}, we will prove the uniqueness of the $\qdist_h$-geodesics (Theorem~\ref{thm:geo_unique}).  Then, in Section~\ref{sec:singularity}, we will analyze the regularity of the $\qdist_h$-geodesics, thus establish Theorems~\ref{thm:singular} and~\ref{thm:regularity}.  Finally, in Section~\ref{sec:conf_remov}, we will prove the removability of the $\qdist_h$-geodesics (Theorem~\ref{thm:removable}). In Appendix~\ref{sec:SLE}, we will estimate the annulus-crossing probabilities for $\SLE$ curves.

Theorem~\ref{thm:singular} (as well as Theorem~\ref{thm:regularity}) will be established by showing that the geodesics in LQG are in a certain sense much more regular than $\SLE$ curves.  In particular, we will show that the probability that a geodesic has four (or more) crossings across an annulus $B(z,\epsilon) \setminus \ol{B(z,\epsilon^\alpha)}$ for $\alpha > 1$ and $\epsilon > 0$ decays significantly more quickly as $\epsilon \to 0$ than for $\SLE_\kappa$ for any value of $\kappa > 0$.

\subsection*{Acknowledgements.} JM was supported by ERC Starting Grant 804166 (SPRS).  WQ acknowledges the support by EPSRC grant EP/L018896/1 and a JRF of Churchill college.  We thank an anonymous referee for helpful feedback on an earlier version of the article.

\section{Preliminaries}
\label{sec:preliminaries}

\subsection{The Gaussian free field}
\label{subsec:gff}

We will now give a brief review of the properties of the Gaussian free field (GFF) which will be important for the present work.  See \cite{s2007gff} for a more in-depth review.

We will first remind the reader how the GFF on a bounded domain is defined before reviewing the definition of the whole-plane GFF.  Suppose that $D \subseteq \C$ is a bounded domain.  We let $C_0^\infty(D)$ be the space of infinitely differentiable functions with compact support contained in $D$.  We define the \emph{Dirichlet inner product} by
\begin{equation}
\label{eqn:dirichlet}
(f,g)_\nabla = \frac{1}{2\pi} \int \nabla f(x) \cdot \nabla g(x) dx \quad\mathrm{for}\quad f,g \in C_0^\infty(D).
\end{equation}
We let $\| \cdot \|_\nabla$ be the corresponding norm.  The space $H_0^1(D)$ is the Hilbert space completion of $C_0^\infty(D)$ with respect to $(\cdot,\cdot)_\nabla$.  Suppose that $(\phi_n)$ is an orthonormal basis of $H_0^1(D)$ and that $(\alpha_n)$ is an i.i.d.\ sequence of $N(0,1)$ variables.  Then the Gaussian free field (GFF) $h$ on $D$ is defined by
\begin{equation}
\label{eqn:gff_series}
h = \sum_{n=1}^\infty \alpha_n \phi_n.
\end{equation}
Since the partial sums for $h$ a.s.\ diverge in $H_0^1(D)$, one needs to take the limit in a different space (e.g., the space of distributions).

In this work, we will be mainly focused on the whole-plane GFF (see \cite[Section~3.2]{s2016zipper} for a review).  To define it, we replace $H_0^1(D)$ with the closure with respect to $(\cdot,\cdot)_\nabla$ of the functions in $C^{\infty}_{0}(\C)$ whose gradients are in $L^2(\C)$, viewed modulo additive constant.  The whole-plane GFF is then defined using a series expansion as in~\eqref{eqn:gff_series} except the limit is taken in the space of distributions modulo additive constant.  This means that if $h$ is a whole-plane GFF and $\phi \in C_{0}^\infty(\C)$ has mean-zero (i.e., $\int \phi(z) dz = 0$) then $(h,\phi)$ is defined.  There are various ways of fixing the additive constant for a whole-plane GFF so that one can view it as a genuine distribution.  For example, if $\phi \in C_{0}^\infty(\C)$ with $\int \phi(z) dz = 1$ then we can set $(h,\phi) = 0$.  If $\psi \in C_{0}^\infty(\C)$ with $\int \psi(z) dz = 1$, then we set
\[ (h,\psi) := (h,\psi-\phi) + (h,\phi) = (h,\psi-\phi).\]
Note that $(h,\psi-\phi)$ is well-defined as $\psi-\phi$ has mean zero.  This definition extends by linearity to any choice of $\psi \in C_{0}^\infty(\C)$.  It can also be convenient to fix the additive constant by requiring setting the average of $h$ on some set, for example a circle (see more below), to be equal to $0$.

{\it Circle averages.}  The GFF is a sufficiently regular distribution that one can make sense of its averages on circles.  We refer the reader to \cite[Section~3]{ds2011kpz} for the rigorous construction and basic properties of GFF circle averages.  For $z \in D$ and $\epsilon > 0$ so that $B(z,\epsilon) \subseteq D$ we let $h_\epsilon(z)$ be the average of $h$ on $\partial B(z,\epsilon)$.

{\it Markov property.}  Suppose that $U \subseteq D$ is open.  Then we can write $h = h_1 + h_2$ where $h_1$ (resp.\ $h_2$) is a GFF (resp.\ a harmonic function) in $U$ and $h_1,h_2$ are independent.  This can be seen by noting that $H_0^1(D)$ can be written as an orthogonal sum consisting of $H_0^1(U)$ and those functions in $H_0^1(D)$ which are harmonic on $U$.  The same is also true for the whole-plane GFF except $h_2$ is only defined modulo additive constant.

We emphasize that $h_2$ is measurable with respect to the values of $h$ on $D \setminus U$.  To make this more precise, suppose that $K$ is a closed set and $\delta > 0$.  We then let $\CF_K^\delta$ be the $\sigma$-algebra generated by $(h,\phi)$ for $\phi \in C_0^\infty(D)$ with support contained in the $\delta$-neighborhood of $K$ and then take $\CF_K = \cap_{\delta > 0} \CF_K^\delta$.  Then $h_2$ is $\CF_K$-measurable and $h_1$ is independent of $\CF_K$ with $K = D \setminus U$.

{\it Local sets.}  The notion of a local set of the GFF serves to generalize the Markov property to the setting in which $K = D \setminus U$ can be random, in the same way that stopping times generalize the Markov property for Brownian motion to times which can be random (see \cite{ss2013continuum_contour} for a review).  More precisely, we say that a (possibly random) closed set $K$ coupled with $h$ is local for $h$ if it has the property that we can write $h = h_1 + h_2$ where, given $\CF_K$, $h_1$ is a GFF on $D \setminus K$ and $h_2$ is harmonic on $D \setminus K$.  Moreover, $h_2$ is $\CF_K$-measurable.

{\it Conformal invariance.}  Suppose that $\varphi \colon \wt{D} \to D$ is a conformal transformation.  It is straightforward to check that the Dirichlet inner product~\eqref{eqn:dirichlet} is conformally invariant in the sense that $(f \circ \varphi, g \circ \varphi)_\nabla = (f,g)_\nabla$ for all $f,g \in C_0^\infty(D)$.  As a consequence, the GFF is conformally invariant in the sense that if $h$ is a GFF on $D$ then $h \circ \varphi$ is a GFF on $\wt{D}$.

{\it Perturbations by a function.}  Suppose that $f \in H_0^1(D)$.  Then the law of $h+f$ is the same as the law of $h$ weighted by the Radon-Nikodym derivative $\exp( (h,f)_\nabla - \| f\|_\nabla^2/2)$.  Consequently, the laws of $h+f$ and $h$ are mutually absolutely continuous.  This can be seen by writing $f = \sum_{n=1}^\infty \beta_n \phi_n$ where $(\phi_n)$ is an orthonormal basis of $H_0^1(D)$, noting that the Radon-Nikodym derivative can be written as $\prod_{n=1}^\infty \exp( \alpha_n \beta_n - \beta_n^2/2)$ and weighting the law of $\alpha_n$ by $\exp( \alpha_n \beta_n - \beta_n^2/2)$ is equivalent to shifting its mean by $\beta_n$.

\subsection{The Schramm-Loewner evolution}
\label{subsec:sle}

The Schramm-Loewner evolution $\SLE$ was introduced by Schramm in \cite{s2000scaling} as a candidate to describe the scaling limit of discrete models from statistical mechanics.  There are several different variants of $\SLE$: chordal (connects two boundary points), radial (connects a boundary point to an interior point), and whole-plane (connects two interior points).  We will begin by briefly discussing the case of chordal $\SLE$ since it is the most common variant and the one for which it is easiest to perform computations.  As the different types of $\SLE$'s are locally absolutely continuous (see \cite{sw2005coordinate}), any distinguishing statistic that we identify for one type of $\SLE$ will also work for other types of $\SLE$s.

Suppose that $\eta$ is a simple curve in $\h$ from $0$ to $\infty$.  For each $t \geq 0$, we can let $\h_t = \h \setminus \eta([0,t])$ and $g_t$ be the unique conformal transformation $\h_t \to \h$ with $g_t(z)-z \to 0$ as $z \to \infty$.  Then the family of conformal maps $(g_t)$ satisfy the chordal Loewner equation (provided $\eta$ is parameterized appropriately):
\[ \partial_t g_t(z) = \frac{2}{g_t(z) - U_t},\quad g_0(z) = z.\]
Here, $U \colon [0,\infty) \to \R$ is a continuous function which is given by the image of the tip of~$\eta$ at time~$t$.  That is, $U_t = g_t(\eta(t))$.

$\SLE_\kappa$ for $\kappa \geq 0$ is the random fractal curve which arises by taking $U_t = \sqrt{\kappa} B_t$ where $B$ is a standard Brownian motion.  (It is not immediate from the definition of $\SLE$ that it is in fact a curve, but this was proved in \cite{rs2005basic}.)  It is characterized by the \emph{conformal Markov property}, which states the following.  Let $\CF_t = \sigma( U_s : s \leq t) = \sigma(\eta(s) : s \leq t)$ and $f_t = g_t - U_t$.  Then:
\begin{itemize}
\item Given $\CF_t$, we have that $s \mapsto f_t(\eta(s+t))$ is equal in distribution to $\eta$.
\item The law of $\eta$ is scale-invariant: for each $\alpha > 0$, $t \mapsto \alpha^{-1} \eta(\alpha^2 t)$ is equal in distribution to $\eta$.
\end{itemize}

We recall that the $\SLE$ curves are simple for $\kappa \in (0,4]$, self-intersecting but not space-filling for $\kappa \in (4,8)$, and space-filling for $\kappa \geq 8$ \cite{rs2005basic}.

Since this work is focused on geodesics which connect two interior points, the type of $\SLE$ that we will make a comparison with is the whole-plane $\SLE$.  Whole-plane $\SLE$ is typically defined in terms of the setting in which $0$ is connected to $\infty$ and then for other pairs of points by applying a M\"obius transformation to the Riemann sphere.  Suppose that $U_t = \sqrt{\kappa} B_t$ where $B$ is a two-sided (i.e., defined on $\R$) standard Brownian motion and we let $(g_t)$ be the family of conformal maps which solve
\begin{equation}
\partial_t g_t(z) = g_t(z) \frac{e^{i U_t} + g_t(z)}{e^{i U_t} - g_t(z)}, \quad g_0(z) = z.
\end{equation}
The whole-plane $\SLE_\kappa$ in $\C$ from $0$ to $\infty$ encoded by $U$ is the random fractal curve $\eta$ with the property that for each $t \in \R$, $g_t$ is the unique conformal transformation from the unbounded component of $\C \setminus \eta((-\infty,t])$ to $\C \setminus \ol{\D}$ which fixes and has positive derivative at $\infty$.

We will prove in Appendix~\ref{sec:SLE} the following proposition, which is the precise property that will allow us to deduce the singularity between $\SLE$ and $\qdist_h$-geodesics.

\begin{proposition}
\label{prop:whole_plane_sle_crossings}
Fix $\kappa > 0$.  Suppose that $\eta$ is a whole-plane $\SLE_\kappa$ in $\C$ from $0$ to $\infty$.  For each $n \in \N$ there exists $\alpha > 1$ such that the following is true.  There a.s.\ exists $\epsilon_0 > 0$ so that for all $\epsilon \in (0,\epsilon_0)$ there exists $z \in B(0,2) \setminus \ol{\D}$ such that $\eta$ makes at least $n$ crossings across the annulus $B(z,\epsilon) \setminus \ol{B(z,\epsilon^\alpha)}$.
\end{proposition}

We will in fact deduce Proposition~\ref{prop:whole_plane_sle_crossings} in Appendix~\ref{sec:SLE} from the analogous fact for chordal $\SLE$, by local absolute continuity between the different forms of $\SLE$.

\begin{proposition}
\label{prop:sle_crossings}
Fix $\kappa > 0$.  Suppose that $\eta$ is a chordal $\SLE_\kappa$ in $\h$ from $0$ to $\infty$.  For each $n \in \N$ there exists $\alpha >1$ such that the following is true.  There a.s.\ exists $\epsilon_0 > 0$ so that for all $\epsilon \in (0,\epsilon_0)$ there exists $z \in [-1,1] \times [0,1]$ such that $\eta$ makes at least $n$ crossings across the annulus $B(z,\epsilon) \setminus \ol{B(z,\epsilon^\alpha)}$.
\end{proposition}

\subsection{Distortion estimates for conformal maps}

Here, we recall some of the standard distortion and growth estimates for conformal maps which we will use a number of times in this article.
\begin{lemma}[Koebe-$1/4$ theorem]
\label{lem:koebe_quarter}
Suppose that $D \subseteq \C$ is a simply connected domain and $f \colon \D \to D$ is a conformal transformation.  Then $D$ contains $B(f(0),|f'(0)|/4)$.
\end{lemma}
The following is a corollary of Koebe-$1/4$ theorem, for example see  \cite[Corollary~3.18]{law2005conformally}.
\begin{lemma}
\label{lem:deriv_distance}
Suppose that $D,\wt{D} \subseteq \C$ are domains and $f \colon D \to \wt{D}$ is a conformal transformation.  Fix $z \in D$ and let $\wt{z} = f(z)$.  Then
\[ \frac{\dist(\wt{z},\partial \wt{D})}{4\dist(z,\partial D)} \leq |f'(z)| \leq \frac{4 \dist(\wt{z},\partial \wt{D})}{\dist(z,\partial D)}.\]
\end{lemma}
The following is  a consequence of Koebe-$1/4$ theorem and the growth theorem, for example see \cite[Corollary~3.23]{law2005conformally}.
\begin{lemma}
\label{lem:points_distance}
Suppose that $D,\wt{D} \subseteq \C$ are domains and $f \colon D \to \wt{D}$ is a conformal transformation.  Fix $z \in D$ and let $\wt{z} = f(z)$.  Then for all $r\in(0,1)$ and all $|w-z|\le r\dist(z, \partial D)$,
\begin{align*}
|f(w)-\wt z| \le \frac{4|w-z|}{1-r^2} \frac{\dist(\wt z, \partial \wt D)}{\dist(z, \partial D)} \le \frac{4r}{1-r^2}\dist(\wt z, \partial \wt D).
\end{align*}
\end{lemma}

\subsection{Binomial concentration}

We will make frequent use of the following basic concentration inequality for binomial random variables.

\begin{lemma}
\label{lem:binomial}
Fix $p \in (0,1)$ and $n \in \N$ and let $X$ be a binomial random variable with parameters $p$ and $n$.  For each $r \in (p,1)$ we have that
\begin{equation}
\label{eqn:binomial_ubd}
\p[ X \geq r n] \leq \left(\frac{1-p}{1-r}\right)^{n(1-r)} \left(\frac{p}{r} \right)^{nr} = \exp(-c_{p,r} n).
\end{equation}
Similarly, for each $r \in (0,p)$ we have that
\begin{equation}
\label{eqn:binomial_lbd}
\p[ X \leq rn ] \leq \left(\frac{1-p}{1-r}\right)^{n(1-r)} \left(\frac{p}{r} \right)^{nr} = \exp(-c_{p,r} n).
\end{equation}
\end{lemma}
We emphasize that for fixed $r$, $c_{p,r} \to \infty$ as $p \to 0$ and also as $p \to 1$.
\begin{proof}
We will prove~\eqref{eqn:binomial_ubd}.  The proof of~\eqref{eqn:binomial_lbd} follows by replacing $X$ with $n-X$, $p$ with $1-p$, and $r$ with $1-r$.  For each $\lambda > 0$, we have that
\[ \p[ X \geq r n] \leq e^{- \lambda r n} \E[ e^{\lambda X} ] = (1-p + p e^{\lambda})^n  e^{-\lambda rn}.\]
Optimizing over $\lambda > 0$ implies~\eqref{eqn:binomial_ubd}.
\end{proof}

\section{Uniqueness: Proof of Theorem~\ref{thm:geo_unique}}
\label{sec:uniqueness}

\begin{figure}[ht!]
\begin{center}
\includegraphics[width=0.8\textwidth]{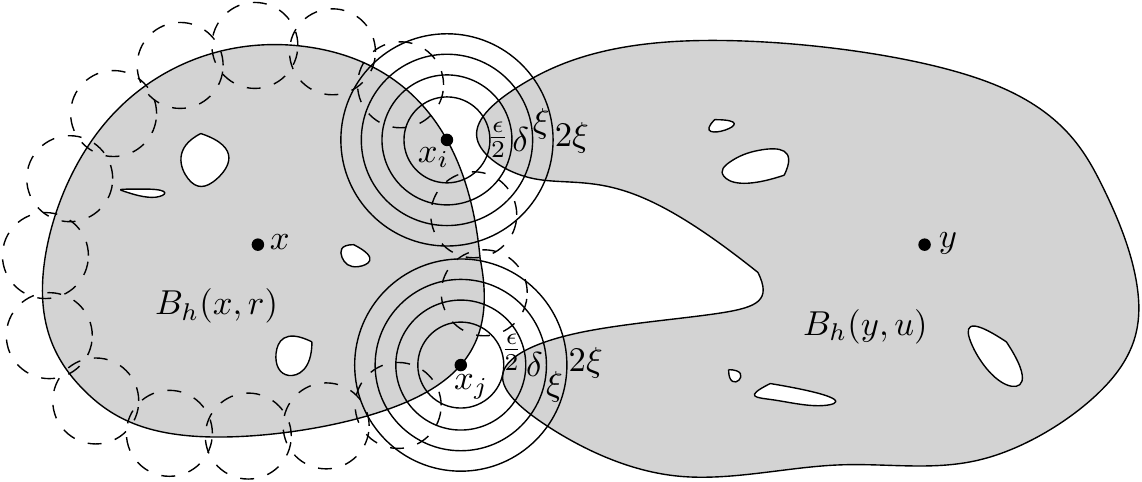}
\end{center}
\caption{\label{fig:uniqueness} Illustration of the proof of Theorem~\ref{thm:geo_unique}. The quantum balls $B_h(x,r)$ and $B_h(y,u)$ are drawn in grey. We cover $\partial B_h(x,r)$ by balls of radius $\eps/2$. We investigate the probability that there are two geodesics from $x$ to $y$ that respectively intersect $B(x_i,\eps/2)$ and $B(x_j, \eps/2)$ such that $B(x_i, 2\xi)\cap B(x_j, 2\xi)=\emptyset$.}
\end{figure}

See Figure~\ref{fig:uniqueness} for an illustration. Fix $x,y \in \C$ distinct.  
For any $r > 0$, let $B_h(x,r)$ be the $\qdist_h$ metric ball centered at $x$ of radius $r$ and let $s:= \inf\{ t > 0 : B_h(x,r) \cap B_h(y,t) \neq \emptyset\}$. 
Note that if $r<\qdist_h(x,y)$, then $s=\qdist_h(x,y)-r$.
To prove the theorem, it suffices to show that for any $r>0$, on the event $\{r<\qdist_h(x,y)\}$, $\partial B_h(x,r) \cap \partial B_h(y,s)$ a.s.\  contains a unique point.  
Indeed, if  $\eta$ is a geodesic from $x$ to $y$, then we can continuously parameterize $\eta$ by $t\in [0, \qdist_h(x,y)]$ so that $\qdist_h (\eta(t),x)=t$, since $\qdist_h$ is homeomorphic to the Euclidean metric.
In particular, for all $r\in[0,\qdist_h(x,y)]$, we have $\eta(r)\in \partial B_h(x,r) \cap \partial B_h(y,s)$.
If for every $r>0$, on the event $\{r<\qdist_h(x,y)\}$,  $\partial B_h(x,r) \cap \partial B_h(y,s)$ a.s.\ contains a unique point, then for any two geodesics $\eta$ and $\wt{\eta}$  from $x$ to $y$,  we a.s.\ have that $\eta(r) = \wt{\eta}(r)$ for all rational $r\in [0, \qdist_h(x,y)]$ simultaneously.  This can only be the case if we a.s.\ have that $\eta = \wt{\eta}$.

From now on, fix $r, \xi>0$.  We will argue that on the event $\{r<\qdist_h(x,y)\}$, $\partial B_h(x,r) \cap \partial B_h(y,s)$ a.s.\ does not contain points which have distance more than $8 \xi$ from each other.  This will imply the desired result as we have taken $r, \xi > 0$ to be arbitrary.  
For $R,\epsilon,\delta > 0$, we define  $E(R,\epsilon,\delta)$ to be the event that
\begin{enumerate}[(i)]
\item \label{itmE1} $ B_h(x,\qdist_h(x,y)) \cup B_h(y,\qdist_h(x,y)) \subseteq B(0,R)$;
\item \label{itmE2} for all $z \in B(0,R)$, the $\qdist_h$-diameter of $B(z,\eps)$ is at most equal to the infimum of  $\qdist_h(a,b)$ over all $a, b \in B(0,R)$ with $|a-b|\ge\delta/2$;
\item \label{itmE3} for all $a, b \in B(0,R)$ with $|a-b| \leq 2\delta$, any geodesic from $a$ to $b$ has Euclidean diameter at most~$\xi$.
\end{enumerate}
Since we have assumed that $\qdist_h$ induces the Euclidean topology, it follows that the probability of~\eqref{itmE1} tends to $1$ as $R \to \infty$.  For the same reason,  for fixed $R$ and $\xi$, as $\delta\to 0$, the probability of~\eqref{itmE3} tends to $1$. Moreover, for fixed $R$ and $\delta$, as $\eps\to 0$, the probability of~\eqref{itmE2} tends to $1$.  Therefore, we can choose $ R,\eps,\delta$ in a way such that $\eps<\delta<\xi$ and the probability of $E(R,\eps,\delta)$ is arbitrarily close to~$1$.

Let $x_1,\ldots,x_n$ be a collection of points on $\partial B_h(x,r)$ so that $\partial B_h(x,r) \subseteq \cup_{j=1}^n B(x_j,\eps/2)$.  We aim to prove that, conditionally on $\{r<\qdist_h(x,y)\}\cap E(R,\eps,\delta)$, there a.s.\ do not exist two geodesics $\eta$ and $\wt\eta$  from $x$ to $y$ such that $\eta$ intersects $B(x_i,\eps/2)$ and $\wt\eta$ intersects $B(x_j,\eps/2)$, where $i,j\in[1,n]$  are such that $B(x_i,2\xi) \cap B(x_j,2\xi) = \emptyset$. This implies that any two intersection points of $\partial B_h(x,r)$ and $\partial B(y,s)$ must have distance at most $8\xi$ from each other. Since the probability of $E(R,\eps,\delta)$ can be made arbitrarily close to $1$, this will complete the proof.

 From now on, we further fix $R,\eps,\delta$ and work on the event  $E:=\{r<\qdist_h(x,y)\}\cap E(R,\eps,\delta)$.  We also assume that the additive constant for $h$ is fixed so that its average on $\partial B(R+2,1)$ is equal to~$0$ (recall that the $\qdist_h$-geodesics do not depend on the choice of additive constant; the choice here is made so that the circle is disjoint from $B(0,R)$ but is otherwise arbitrary).
Fix $i,j\in[1,n]$ such that $B(x_i,2\xi) \cap B(x_j, 2\xi) = \emptyset$. 
Let $u:=\inf\{t>0: B_h(y,t)\cap B(x_i,\eps/2)\not=\emptyset \text{ and }  B_h(y,t)\cap B(x_j,\eps/2)\not=\emptyset\}$.
If $u \ge s$, then obviously there do not exist two geodesics $\eta$ and $\wt\eta$  from $x$ to $y$ such that $\eta$ intersects $B(x_i,\eps/2)$ and $\wt\eta$ intersects $B(x_j,\eps/2)$.

On $E\cap\{u<s\}$, for any $\ell \in \{i,j\}$, due to~\eqref{itmE2}, the $\qdist_h$-shortest path from $\partial B_h(y, u)$ to $\partial B_h(x,r) \cap B(x_\ell, \eps/2)$ must have one endpoint in $\partial B_h(y,u) \cap B(x_\ell, \delta)$ (and the other endpoint is in $\partial B_h(x,r) \cap B(x_\ell, \eps/2)$).
For $\ell \in \{i,j\}$, we let $X_\ell$ be the infimum of $\qdist_h$-lengths of paths which connect a point on $\partial B_h(x,r) \cap B(x_\ell,\delta)$ to a point on $\partial B_h(y,u) \cap B(x_\ell,\delta)$.  We are going to prove that on $E\cap\{u<s\}$, we have $X_i \neq X_j$ a.s.  
This will imply that, on the event $E$, there a.s.\ do not exist two geodesics $\eta$ and $\wt\eta$  from $x$ to $y$ such that $\eta$ intersects $B(x_i,\eps/2)$ and $\wt\eta$ intersects $B(x_j,\eps/2)$, which will complete the proof.

Let us now work on $E\cap\{u<s\}$. We will further condition on the sets $\ol{B_h(x,r)}$ and $\ol{B_h(y,u)}$ (which are local for $h$ by Assumption~\ref{asump:main}). It suffices to show that under such conditioning, $X_i \neq X_j$ a.s.  
On $E(R,\eps,\delta)$, due to~\eqref{itmE3}, for $\ell \in \{i,j\}$, any geodesic which connects a point on $\partial B_h(x,r) \cap B(x_\ell,\delta)$ to a point on $\partial B_h(y,u) \cap B(x_\ell,\delta)$ is contained in $B(x_\ell, \xi)$.  
By the locality of $\qdist_h$, $X_\ell$ is determined by $B_h(x,r)$, $B_h(y,u)$, and the values of $h$ in $B(x_\ell,\xi)$.  Let $\phi$ be a non-negative $C_0^\infty(\C)$ function with support contained in $U_i = B(x_i,2\xi) \setminus (B_h(x,r) \cup B_h(y,u))$ with the property that every path from $\partial B_h(x,r) \cap B(x_i,\delta)$ to $\partial B_h(y,u) \cap B(x_i,\delta)$ contained in $B(x_i,\xi)$ must pass through $\phi^{-1}(\{1\})$.  
We emphasize that we can choose $\phi$ as a deterministic function of $B_h(x,r)$, $B_h(y,u)$ and $x_i, x_j, \xi$.  
For $\alpha \in \R$, we let $X_i^\alpha$ be the infimum of $\qdist_{h+\alpha \phi}$-lengths of paths which connect a point on $\partial B_h(x,r) \cap B(x_i,\delta)$ to a point on $\partial B_h(y,u) \cap B(x_i,\delta)$ and which are contained in $B(x_i,\xi)$.  We note that $X_i^0 = X_i$.  Observe that $X_i^\alpha$ is strictly increasing and continuous in $\alpha$ by part~\eqref{it:scaling} of Assumption~\ref{asump:main}.  Thus if we take $A$ to be uniform in $[0,1]$ then the probability that $X_i^A = X_j$ is equal to $0$.  Since the conditional law of $h+A \phi$ in $U_i$ given the values of $h$ outside of $U_i$ is mutually absolutely continuous with respect to the conditional law of $h$ in $U_i$ given its values outside of $U_i$, it follows that the joint law of $(X_i^A,X_j)$ is mutually absolutely continuous with respect to the joint law of $(X_i,X_j)$.  In particular, the probability that $X_i = X_j$ is also equal to $0$. \qed

\section{Regularity}
\label{sec:singularity}

In this section, we will give the proofs of Theorems~\ref{thm:singular} and~\ref{thm:regularity}.  The first step is carried out in Section~\ref{subsec:good_scales}, which is to show that (with high probability) the whole-plane GFF at an arbitrarily high fraction of geometric scales exhibits behavior (modulo additive constant) which is comparable to the GFF with zero boundary conditions.  We will then use this fact in Section~\ref{subsec:annulus} to show that (with high probability):
\begin{itemize}
\item At an arbitrarily high fraction of geometric scales (depending on a choice of parameters), the shortest path which goes around an annulus is at most a constant times the length of the shortest path which crosses an annulus (Proposition~\ref{prop:dist_around}) and that
\item There exists a geometric scale at which the former is strictly shorter than the latter (consequence of Lemma~\ref{lem:dist_around_good}).
\end{itemize}

The first statement is the main ingredient in the proofs of Theorems~\ref{thm:singular} and~\ref{thm:regularity} since it serves to rule out a geodesic making multiple crossings across annuli.  The second statement will be used to prove an upper bound for the dimension of the geodesics  (Proposition~\ref{prop:dimension_bound}) which will be used in the proof of Theorem~\ref{thm:removable} in Section~\ref{sec:conf_remov}.

Throughout, we let $h$ be a whole-plane GFF. For any $z \in \C$ and $r > 0$,  let $\CF_{z,r}$ be the $\sigma$-algebra generated by the values of $h$ outside of $B(z,r)$. 
By the Markov property for the GFF, we can write~$h$ as a sum of a GFF $\wt{h}_{z,r}$ on $B(z,r)$ with zero boundary conditions and a distribution $\Fh_{z,r}$ which is harmonic on $B(z,r)$ and agrees with $h$ outside of $B(z,r)$. Note that  $\Fh_{z,r}$ is measurable w.r.t.\ $\CF_{z,r}$ and $\wt{h}_{z,r}$ is independent of $\CF_{z,r}$. 
Let  $h_r(z)$ be the average of $h$ on $\partial B(z,r)$. Note that $\Fh_{z,r}(z) = h_r(z)$ since $\Fh_{z,r}$ is harmonic in $B(z,r)$.
Let $\wh h_{z,r}:=h-h_r(z)$. 

\subsection{Good scales}
\label{subsec:good_scales}

In this subsection, we will first define the $M$-good scales and show in  Lemma~\ref{lem:good_scale_rn} that they are important because on such scales the law of a whole-plane GFF and the law of a GFF with zero boundary conditions are mutually absolutely continuous with well-controlled Radon-Nikodym derivatives. 
Then we will prove the main result of this subsection, which is Proposition~\ref{prop:good_scale_density}, which says that an arbitrarily large fraction of scales are $M$-good with arbitrarily large probability provided we choose $M$ large enough.

Fix a constant $M > 0$.  Fix $z \in \C$ and $r > 0$.  We say that $B(z,r)$ is $M$-good for $h$ if:
\[ \sup_{w \in B(z,15r/16)} | \Fh_{z,r}(w) - \Fh_{z,r}(z)| \leq M.\]
Let $E_{z,r}^M$ be the event that $B(z,r)$ is $M$-good and note that $E_{z,r}^M$ is $\CF_{z,r}$-measurable.

\begin{lemma}
\label{lem:good_scale_rn}
Fix $z \in \C$ and $r > 0$. The conditional law given $\CF_{z,r}$ of $\wh{h}_{z,r}$ restricted to $B(z,7r/8)$ is mutually absolutely continuous w.r.t.\ the law of a  zero-boundary GFF on $B(z,r)$  restricted to $B(z,7r/8)$.  

Let $\CZ_{z,r}(\cdot)$ (resp.\ $\CW_{z,r}(\cdot)$) be  the Radon-Nikodym derivative of the former w.r.t.\ the latter (resp.\ latter w.r.t.\ the former).  (Note that $\CZ_{z,r}$ (resp.\ $\CW_{z,r})$  is itself measurable w.r.t.\ $\CF_{z,r}$ and takes as argument $\wt h_{z,r}|_{B(z,7/8)}$ (resp.\ $\wh h_{z,r}|_{B(z,7/8)}$).) On $E_{z,r}^M$, for all $p\in\R$, there exists a constant $c(p,M)$ depending only on $p$ and $M$ such that 
\begin{align*}
\E[\CZ_{z,r}(\wt h_{z,r}|_{B(z,7/8)})^p \giv \CF_{z,r}]\le c(p,M) \quad \text{and} \quad \E[ (\CW_{z,r}(\wh h_{z,r}|_{B(z,7/8)}))^{p} \giv \CF_{z,r}]\le c(p,M) \quad \text{a.s.}
\end{align*}
\end{lemma}

Note that $\E[\CZ_{z,r}(\wt h_{z,r}|_{B(z,7r/8)})^p \giv \CF_{z,r}]$ and  $\E[(\CW_{z,r}(\wh h_{z,r}|_{B(z,7r/8)}))^{p} \giv \CF_{z,r}]$ are both measurable w.r.t.\ $\CF_{z,r}$.

\begin{proof}[Proof of Lemma~\ref{lem:good_scale_rn}]
Note that when restricted to $B(z,r)$, $\wh h_{z,r}$ admits the Markovian decomposition $\wh h_{z,r}=\wt h_{z,r}+ \wh\Fh_{z,r}$ where $\wh\Fh_{z,r}=\Fh_{z,r}-h_r(z)$ is harmonic in $B(z,r)$.
Fix $\phi \in C_0^\infty(B(z,29r/32))$ with $\phi|_{B(z,7r/8)} \equiv 1$ and let $g = \wh{\Fh}_{z,r} \phi$.  Then $\wt{h}_{z,r} + g$ is equal to $\wh{h}_{z,r}$ in $B(z,7r/8)$.  Moreover, if we take the law of $\wt{h}_{z,r}$ and then weight it by the Radon-Nikodym derivative $\CZ_{z,r}^0(\wt h_{z,r})= \exp((\wt{h}_{z,r},g)_{\nabla} - \| g \|_\nabla^2/2)$, then the resulting field has the same law as $\wt{h}_{z,r}+g$.  Therefore $\CZ_{z,r}$ is given by integrating $\CZ_{z,r}^0$ over the randomness of $\wt{h}_{z,r}$ in $B(z,r) \setminus B(z,7r/8)$ given $\CF_{z,r}$.  
Conversely, if we take the law of $\wt h_{z,r}+g$ and weight it by the Radon-Nikodym derivative 
\begin{align}\label{eq:radon-nikodym}
 \exp((\wt{h}_{z,r}+g,-g)_{\nabla} + \| g \|_\nabla^2/2)=\exp((\wt{h}_{z,r},-g)_{\nabla} - \| g \|_\nabla^2/2)=\exp((\wh{h}_{z,r},-g)_{\nabla} - \| g \|_\nabla^2/2),
 \end{align}
then the resulting field has the same law as $\wt h_{z,r}$.

Note that the second equality in~\eqref{eq:radon-nikodym} holds because $\wt{h}_{z,r}$ differs from $\wh{h}_{z,r}$ by a function which is harmonic in $B(z,r)$ and $g$ is supported in $B(z,r)$.
Since $\wt h_{z,r}+g$ and $\wh h_{z,r}$ agree on $B(z,7r/8)$, we get that if we take the law of $\wh h_{z,r}$ and weight it by $\CW_{z,r}^0=\exp((\wh{h}_{z,r},-g)_{\nabla} - \| g \|_\nabla^2/2)$, then the restriction of the resulting field to $B(z,7r/8)$ has the same law as the corresponding restriction of $\wt h_{z,r}$. 
 Therefore $\CW_{z,r}$ is given by integrating $\CW_{z,r}^0$ over the randomness of $\wh{h}_{z,r}$ in $B(z,r) \setminus B(z,7r/8)$ given $\CF_{z,r}$.  This proves the mutual absolute continuity.

Now suppose that we are working on the event $E_{z,r}^M$.  Then $|\wh{\Fh}_{z,r}| \leq M$ in $B(z,15r/16)$.  Recall the following basic derivative estimate for harmonic functions.  There exists a constant $c > 0$ so that if $R > 0$ and $u$ is harmonic in $B(z,R)$ then for $w \in B(z,R)$ we have that
\begin{equation}
\label{eqn:harmonic_function_estimate}
|\nabla u(w)| \leq c(\dist(w,\partial B(z,R))^{-1} \sup_{v \in B(z,R)}|u(v) - u(z)|.	
\end{equation}
Applying this with $u = \wh{\Fh}_{z,r}$, $R = 15r/16$, and $w \in B(z,29r/32)$ we see that $\| \wh{\Fh}_{z,r} \|_\nabla^2$ (with the norm computed on $B(z,29r/32)$) is bounded by a constant which depends only on $M$.  Therefore the same is true for $\| g \|_\nabla^2$.  The second part of the lemma follows because for all $p\in\R$,
\begin{align}\label{eq:radon}
\E\!\left[(\CZ_{z,r}^{0}(\wt h_{z,r}))^p \giv \CF_{z,r} \right]=\E[(\CW_{z,r}^0(\wh h_{z,r}))^p \giv \CF_{z,r}]=\exp( (p^2-p) \| g \|_\nabla^2/2).
\end{align} 
In particular, on $E_{z,r}^M$, the above quantities are bounded by a constant which depends only on~$p$ and~$M$.  The same is therefore true for $\CZ_{z,r}$ and $\CW_{z,r}$ by Jensen's inequality, which completes the proof.
\end{proof}

Now let us mention a few consequences of this lemma and its proof that we will use later on.
\begin{remark}
\label{rmk1}
Fix $p > 1$ and let $q > 1$ be such that $p^{-1} + q^{-1} = 1$.  For any  GFF $h^0$ defined on $B(z,7r/8)$, let $E(h^0)$ be an event which is determined by $h^0$.  Then Lemma~\ref{lem:good_scale_rn} combined with H\"older's inequality implies that there exist constants $c_1(p,M), c_2(p,M)$ depending only on $p,M$ so that on $E_{z,r}^M$ we have
\begin{align}
 \label{eqn:rn_moment_bound}
& \p[ E(\wh h_{z,r}|_{B(z,7r/8)}) \giv \CF_{z,r}] =\p \!\left[\CZ_{z,r}(\wt h_{z,r}|_{B(z,7r/8)}) \mathbf{1}_{E(\wt h_{z,r}|_{B(z,7r/8)})} \giv \CF_{z,r} \right] \\
\notag
 &\leq \E\!\left[\!\left(\CZ_{z,r}(\wt h_{z,r}|_{B(z,7r/8)}) \right)^{p} \giv \CF_{z,r} \right]^{1/p} \p[E(\wt h_{z,r}|_{B(z,7r/8)})]^{1/q} 
  \le c_1(p,M) \p [E(\wt h_{z,r}|_{B(z,7r/8)})]^{1/q},\\
 \label{eqn:rn_moment_bound2}
& \p[E(\wt h_{z,r}|_{B(z,7r/8)})]=\E\!\left[\CW_{z,r}(\wh h_{z,r}|_{B(z,7r/8)}) \mathbf{1}_{E(\wh h_{z,r}|_{B(z,7r/8)})} \giv \CF_{z,r} \right] \\
 \notag
&\le \E\!\left[\CW_{z,r}(\wh h_{z,r}|_{B(z,7r/8)})^{p} \giv \CF_{z,r}\right]^{1/p} \p\!\left[E(\wh h_{z,r}|_{B(z,7r/8)}) \giv \CF_{z,r} \right]^{1/q} 
\le c_2(p,M)  \p\!\left[E(\wh h_{z,r}|_{B(z,7r/8)}) \giv \CF_{z,r} \right]^{1/q}. 
\end{align}
\end{remark}

Now let us show the main result of this subsection.

\begin{proposition}
\label{prop:good_scale_density}
Fix $z \in \C$ and $r > 0$.  For each $k \in \N$, we let $r_k = 2^{-k} r$.  Fix $K \in \N$ and let $N = N(K,M)$ be the number of $1 \leq k \leq K$ so that $B(z,r_k)$ is $M$-good.  For every $a > 0$ and $b\in (0,1)$ there exists $M_0 = M(a,b)$ and $c_0(a,b)$, so that for all $M \geq M_0$ we have
\[ \p[ N(K,M) \leq b K] \leq c_0(a,b)e^{-a K}.\]
\end{proposition}

One main input into the proof of Proposition~\ref{prop:good_scale_density} is the following bound for the probability that a given ball is not $M$-good.

\begin{lemma}
\label{lem:good_scale}
There exist constants $c_1,c_2 > 0$ such that for any $z \in \C$, $r > 0$, and $M>0$, we have
\[ \p\!\left[ (E_{z,r}^M)^c\right] \leq c_1 e^{-c_2 M^2}.\]
\end{lemma}
\begin{proof}
By the scale and translation invariance of the whole-plane GFF, the quantity $\p\!\left[ (E_{z,r}^M)^c\right]$ is independent of $z$ and $r$, hence we will choose $z=0$ and $r=1$.
We are going to bound the supremum of $|\Fh_{0,1}(w) - \Fh_{0,1}(0)|$ when $w \in B(0,15/16)$ and show that it has a Gaussian tail.  

Let $\Fp$ be the Poisson kernel on $B(0,31/32)$.  Then there exists an absolute constant $C > 0$ so that $\Fp(w,y) \leq C$ for all $w \in B(0,15/16)$ and $y\in\partial B(0,31/32)$.  Letting $dy$ denote the uniform measure on $\partial B(0,31/32)$, we have that for all $w\in  B(0,15/16)$
\begin{align*}
   |\Fh_{0,1}(w) - \Fh_{0,1}(0)|
&= \left| \int_{\partial B(0,31/32)} (\Fh_{0,1}(y)- \Fh_{0,1}(0)) \Fp(w,y) dy \right|\\
&\leq \int_{\partial B(0,31/32)} |\Fh_{0,1}(y) - \Fh_{0,1}(0)| \Fp(w,y) dy\\
&\leq C \int_{\partial B(0,31/32)} |\Fh_{0,1}(y) - \Fh_{0,1}(0)| dy.
\end{align*}
Therefore by Jensen's inequality, we have that
\[ \exp\left( a \sup_{w \in B(0,15/16)} |\Fh_{0,1}(w) - \Fh_{0,1}(0)|^2 \right) \leq \int_{\partial B(0,31/32)} e^{a C^2 |\Fh_{0,1}(y)- \Fh_{0,1}(0)|^2} dy.\]
We note that $\Fh_{0,1}(y) - \Fh_{0,1}(0)$ is a Gaussian random variable with bounded mean and variance.  It thus follows that by choosing $a > 0$ sufficiently small we have 
\[ \E\left[ \exp\left( a \sup_{w \in B(0,15/16)} |\Fh_{0,1}(w) - \Fh_{0,1}(0)|^2 \right) \right]  < \infty.\]
The result therefore follows by Markov's inequality.
\end{proof}
\begin{remark}
The same reasoning applies to the zero-boundary GFF. Let $\wt h$ be a zero-boundary GFF in $B(0,1)$. For all $r \in(0,1)$, let $\wt\Fh_{0,r}$ be the field which is harmonic in $B(0,r)$ and agrees with $\wt h$ in $B(0,1)\setminus \ol{B(0,r)}$.  We can similarly deduce that there exist $c_1, c_2>0$ such that for all $r\in(0,1)$ and $M > 0$ we have
\begin{align}\label{eq:zero-boundary}
\p\!\left[\sup_{w\in B(0,15r/16)}|\wt\Fh_{0,r}(w)-\wt\Fh_{0,r}(0)|>M\right]\le c_1 e^{-c_2 M^2}.
\end{align}
\end{remark}

\begin{proof}[Proof of Proposition~\ref{prop:good_scale_density}]
By the translation and scale-invariance of the whole-plane GFF, the statement is again independent of $z$ and $r$, hence we will choose $z=0$ and $r=1$ so that $r_k=2^{-k}$.
Our strategy is to explore $h$ in a Markovian way from outside in and to control (using Lemma~\ref{lem:good_scale}) the number of scales we need to go in each time in order to find the next $M$-good scale.

We start by looking for the first $k_0\in\N$ for which $B(0,r_{k_0})$ is an $M$-good scale.
Let
\[ R = \sup_{w \in B(0,15/16)} |\Fh_{0,1}(w) - \Fh_{0,1}(0)|.\]
Lemma~\ref{lem:good_scale} implies that there is a positive probability $p_M$ that $R \leq M$. In this case, we have $k_0=0$.
With probability $1-p_M$, one has $R > M$. In this case, conditionally on $\CF_{0,1}$ and on $\{R>M\}$ (which is measurable w.r.t.\ $\CF_{0,1}$), we continue to look for the first $k_0\ge 1$ for which $B(0,r_{k_0})$ is an $M$-good scale.
For some $C>0$ that we will adjust later, we aim to find  $\ell\in\N$ such that 
\begin{align}\label{C}
\sup_{w \in B(0,r_\ell)} |\Fh_{0,1}(w) - \Fh_{0,1}(0)| \leq C,
\end{align}
and then to estimate the goodness of the scale $B(0,r_\ell)$.
By applying the derivative estimate~\eqref{eqn:harmonic_function_estimate} to the harmonic function $\Fh_{0,1}$ we see that there exists $c>0$ such that if we choose $\ell= c \lceil \log_2(R) \rceil$, then~\eqref{C} is satisfied.
Lemma~\ref{lem:good_scale} implies that $\p[R > t] \leq c_1 e^{- c_2 t^2}$ for constants $c_1,c_2 > 0$.  Consequently,
\[ \p[ \ell \geq q ] \leq \p[ \log_2(R)  \geq q/c-1] \leq \p[ R \geq 2^{q/c} ] \leq c_1 \exp(-c_2 2^{2 q/c}).\]

Now let us estimate the following quantity, which represents how good $B(0,r_\ell)$ is:
\[ \wh R = \sup_{w \in B(0,15 r_\ell/16)} |\Fh_{0,r_\ell}(w) - \Fh_{0,r_\ell}(0)|.\]
Note that $\Fh_{0,r_\ell}(w)= \Fh_{0,1}(w) + \wt\Fh_{0, r_\ell}(w)$, where $\wt\Fh_{0,r_\ell}$ is harmonic in $B(0,r_\ell)$ and agrees with a zero-boundary GFF in $B(0,1)$ outside of $B(0,r_{\ell})$.
Therefore, combining with~\eqref{C}, we have that
\begin{align}\label{wh_R_bound}
\wh R \le   \sup_{w \in B(0,15 r_\ell/16)} |\wt\Fh_{0,r_\ell}(w) - \wt\Fh_{0,r_\ell}(0)|+C.
\end{align}

Note that $\wt\Fh_{0,r_\ell}$ is independent of $\CF_{0,1}$.
Applying~\eqref{eq:zero-boundary}--\eqref{wh_R_bound}, we know that there exist $\wh c_1, \wh c_2>0$ (depending only on $C$) such that $\p[\wh R > t \mid \CF_{0,1}] \mathbf{1}_{R>M} \leq \wh c_1 e^{- \wh c_2 t^2}$. 
In particular, it implies that the conditional probability of $\wh R \le M$ is at least some $p_{M,C}>0$. We emphasize that $ p_{M,C}$ depends only on $M$ and $C$ and  can be made arbitrarily close to $1$ if we fix $C>0$ and choose $M>0$ sufficiently large.  From now on, we will fix $C$ and reassign the values of $c_1, c_2, p_M, \wh c_1, \wh c_2, p_{M,C}$ so that $\wh c_1=c_1, \wh c_2=c_2, p_{M,C}=p_M$.

If $B(0,r_\ell)$ is $M$-good, then $k_0=\ell$.
Otherwise we continue our exploration, conditionally on $\CF_{0,r_\ell}$ and on the event $\{R>M\} \cap \{\wh R>M\}$ (which is measurable w.r.t.\ $\CF_{0,r_\ell}$). 
Similarly to~\eqref{C}, we define $\wh\ell= c\lceil \log_2(\wh R) \rceil$ so that 
\[ \sup_{w \in B \left(0,r_{\ell+\wh\ell}\right)} |\Fh_{0,r_\ell}(w) - \Fh_{0,r_\ell}(0)| \leq C.\]
Therefore, the goodness of  $B(0,r_{\ell+\wh\ell})$ has the same tail bound as $\wh R$. Hence we know that the probability that $B(0,r_{\ell+\wh\ell})$ is $M$-good (i.e., $k_0=\ell+\wh \ell$) is also at least $p_{M,C}$ and that otherwise we can look at the next scale $B\!\left(0, r_{\ell+2\wh\ell}\right)$. We can thus iterate.

 The above procedure implies that
 \begin{align*}
k_0 \leq \sum_{i=1}^G A_i ,
\end{align*}
 where the $A_i$'s are i.i.d.\ random variables with $\p[ A_i \geq t] \leq c_1 e^{- c_2 t^2}$ and $G$ is a geometric random variable with success probability $p_{M}$. Moreover,  the $A_i$'s and $G$ are all independent.  It thus follows that $k_0$ has an exponential tail.  Indeed,
\begin{align*}
\E\!\left[e^{\lambda k_0} \right]  \leq \sum_{n=1}^\infty \E\!\left[e^{\lambda \sum_{i=1}^n A_i}\right] \p[G=n] = \sum_{n=1}^\infty \E\!\left[ e^{\lambda A_1} \right]^n (1-p_{M})^{n-1} p_{M}.
\end{align*}
Since $A_1$ has a Gaussian tail, $ \E\!\left[ e^{\lambda A_1} \right]$ is finite for any $\lambda>0$. We also know that $p_{M}$ can be made arbitrarily close to $1$ as $M\to\infty$.
Therefore, for each $\lambda>0$ we can choose $M$ big enough so that
\begin{align}\label{eq:exp_k0}
\E\!\left[ e^{\lambda k_0} \right] < 1.
\end{align}

Once we find the first good scale $k_0$, we can repeat the above procedure to find the next good scale $k_0+k_1$. As a first step, instead of going $c\lceil\log_2 R\rceil$ or $c\lceil\log_2 \wh R\rceil$ further (for $R, \wh R>M$), we just need to go $c\lceil\log_2 M\rceil$ further (and then repeat the same procedure). We therefore get that $k_1$ is stochastically dominated by $k_0$. Moreover, $k_1$ is independent of $k_0$.
Therefore, for any $b\in(0,1)$ and $\lambda>0$, we have
\begin{align}\label{bk}
\p[ N(K,M) \le bK ] \le \p\!\left[ \sum_{i=1}^{bK} k_i \ge K\right],
\end{align}
where the $k_i$'s are i.i.d. and distributed like $k_0$. For any $a>0$, by Markov's inequality, the right hand-side of~\eqref{bk} is less than or equal to
\begin{align*}
e^{-a K} \E\!\left[ \exp(a k_0)\right]^{bK}.
\end{align*}
Then it completes the proof due to~\eqref{eq:exp_k0}.
\end{proof}

\subsection{Annulus estimates}
\label{subsec:annulus}

We now proceed to establish the main estimate which will be used to prove Theorems~\ref{thm:singular} and~\ref{thm:regularity}.

\begin{proposition}
\label{prop:dist_around}
Fix $z \in \C$ and $r > 0$.  For each $k$, we let $r_k = 2^{-k} r$.  We also let $L_{1,k}$ be the infimum of $\qdist_h$-lengths of paths contained in $B(z,7r_k/8) \setminus B(z,r_k/2)$ which separate $z$ from $\infty$ and let $L_{2,k}$ be the $\qdist_h$-distance from $\partial B(z, 7r_k/8)$ to $\partial B(z,r_k/2)$.  Fix $K\in\N$, $c > 0$, and let $N(K, c)$ be the number of $k \in \{1,\ldots, K\}$ with the property that $L_{1,k} \leq c L_{2,k}$.  For each $a_1 > 0$ and $b_1 \in (0,1)$, there exist $c_1(a_1,b_1), c_2(a_1,b_1)>0$ such that for all $c \geq c_1(a_1,b_1)$, we have
\[ \p[ N(K,c) \leq b_1 K ] \le c_2(a_1,b_1) e^{-a_1K}.\]
\end{proposition}

The following lemma is the main input into the proof of Proposition~\ref{prop:dist_around}.

\begin{lemma}
\label{lem:dist_around_good}
Fix $z \in \C$ and $r > 0$.  Let $L_1$ be the infimum of $\qdist_h$-lengths of paths contained within the annulus $B(z,7r/8) \setminus B(z,r/2)$ and which separate $z$ from $\infty$.  Let $L_2$ be the $\qdist_h$ distance from $\partial B(z,7r/8)$ to $\partial B(z,r/2)$.  On $E_{z,r}^M$, for all $q>0$, there exists $c_0>0$ depending only on $M$ such that for all $c>c_0$ and all $z\in\C$ and $r>0$, we have
\begin{equation}
\label{eqn:dist_around}
\p[ L_1 \geq c L_2 \giv \CF_{z,r} ] \le q \quad a.s.	
\end{equation}
Let $S_1$ be the infimum of $\qdist_{h}$-lengths of paths contained in $B(0,7r/8) \setminus B(0,3r/4)$ and which separate~$0$ from~$\infty$.  We also let $S_2$ be the $\qdist_h$ distance from $\partial B(0,5r/8)$ to $\partial B(0,r/2)$.
There exists $p \in (0,1)$ depending only on $M$ so that on $E_{z,r}^M$, for all $z\in\C$ and $r>0$, we have
\begin{equation}
\label{eqn:dist_across}
\p[ S_1 < S_2 \giv \CF_{z,r}] \geq p \quad a.s.
\end{equation}
\end{lemma}
\begin{proof}
By part~\eqref{it:coord_change} of Assumption~\ref{asump:main}, if we apply the LQG coordinate change formula~\eqref{eqn:coord_change} using the transformation $w \mapsto r^{-1}(w-z)$ which takes $B(z,r)$ to $B(0,1)$, then the lengths of the geodesics are preserved.  Therefore, we can take $z=0$ and $r=1$.  Note that the events $\{L_1 \geq cL_2\}$ and $\{S_1< S_2\}$ depend only on the restriction of $\wh h_{0,1}$ to $B(0, 7/8)$, hence we can apply Remark~\ref{rmk1} and deduce that on $E_{0,1}^M$,  it suffices to prove the following statement:
Let $\wt{h}$ be an instance of the GFF on $B(0,1)$ with zero boundary conditions.
\begin{enumerate}[(I)]
  \item Let $\wt{L}_1$ be the infimum of $\qdist_{\wt{h}}$-lengths of paths contained in $B(0,7/8) \setminus B(0,1/2)$ which separate~$0$ from~$\infty$ and let $\wt{L}_2$ be the $\qdist_{\wt{h}}$ distance from $\partial B(0,7/8)$ to $\partial B(0,1/2)$.  Then
\begin{align}\label{eqn:dist_around1}
\p[ \wt L_1 \geq c \wt L_2] \to 0 \, \text{ as } c\to \infty.
\end{align}
	\item Let $\wt{S}_1$ be the infimum of $\qdist_{\wt{h}}$-lengths of paths contained in $B(0,7/8) \setminus B(0,3/4)$ which separate~$0$ from~$\infty$ and let $\wt{S}_2$ be the $\qdist_{\wt{h}}$ distance from $\partial B(0,5/8)$ to $\partial B(0,1/2)$.  Then there exists $p \in(0,1)$ such that
\begin{align}\label{eqn:dist_across1}
\p[ \wt{S}_1 < \wt{S}_2 ] \geq p.
\end{align}
\end{enumerate}
Note that~\eqref{eqn:dist_around1} together with~\eqref{eqn:rn_moment_bound} implies~\eqref{eqn:dist_around} and~\eqref{eqn:dist_across1} together with~\eqref{eqn:rn_moment_bound2} implies~\eqref{eqn:dist_across}.

Since we have assumed that the $\qdist_{\wt{h}}$ metric is a.s.\ homeomorphic to the Euclidean metric, it follows that $\wt{L}_1$ and $\wt{L}_2$ are both a.s.\ positive and finite random variables.  It therefore follows that~\eqref{eqn:dist_around1} holds.

Let us now prove~\eqref{eqn:dist_across1}.   Let $\phi$ be a non-negative, radially symmetric $C_0^\infty$ function supported in $B(0,3/4)$ and which is equal to $1$ in $B(0,5/8)$.  Then adding $c \phi$ to $\wt{h}$ does not affect $\wt{S}_1$ but it multiplies $\wt{S}_2$ by $e^{\beta c}$ where $\beta$ is as in part~\eqref{it:scaling} of Assumption~\ref{asump:main}.  Since $\wt{S}_1$, $\wt{S}_2$ are a.s.\ positive and finite, it follows that by replacing $\wt{h}$ by $\wt{h} + c \phi$ and taking $c > 0$ sufficiently large we will have that $\wt{S}_1 < \wt{S}_2$ with positive probability.  This completes the proof as $\wt{h} + c \phi$ is mutually absolutely continuous w.r.t.\ $\wt h$.
\end{proof}

\begin{proof}[Proof of Proposition~\ref{prop:dist_around}]
Fix $z\in\C$ and $r>0$. Let $E(K,b)$ denote the event that  the fraction of $k \in \{1,\ldots,K\}$ for which $B(z, r_k)$ is $M$-good is at least $b$.
Proposition~\ref{prop:good_scale_density} implies that for any $b\in(0,1)$ and $a>0$, there exists $M > 0$ sufficiently large so that 
\begin{align}\label{eq:EbK}
\p[E(K,b)] = 1-O(e^{-aK}).
\end{align}
 We thereafter fix $a,b$ and $M$ so that~\eqref{eq:EbK} holds.

Let $L_{1,k}, L_{2,k}$ be as in Lemma~\ref{lem:dist_around_good} for $B(z,r_k)$.  Lemma~\ref{lem:dist_around_good} implies that for each $q>0$ there exists $c > 0$ so that at each $M$-good scale $B(z, r_k)$, we have $\p[ L_{1,k} \geq c L_{2,k} \giv \CF_{z, r_k}] \le q$ a.s.  Note that both $L_{1,k}$ and $L_{2,k}$ are measurable w.r.t.\ $\CF_{z, r_{k+1}}$, hence we can explore $h$ according to the filtration $(\CF_{z, r_k})_{k\ge 0}$.  More precisely, if we explore $h$ from outside in, then each time we encounter a new good scale, conditionally on the past, the probability of achieving $\{ L_{1,k} < c L_{2,k}\}$ for that scale is uniformly bounded from below by $1-q$.  For each $k$, let $g_k$ be the index of the $k$th good scale.  It thus follows that the number $\wt{N}(K,c)$ of $k \in \{1,\ldots,b K\}$ that we achieve $\{L_{1,g_k} < c L_{2,g_k}\}$ is at least equal to a binomial random variable with success probability $1-q$ and $b K$ trials.  By Lemma~\ref{lem:binomial}, this proves that for any $b_1\in(0,b)$ and $\wt a>0$, if we make $q >0$ sufficiently small and $a$ sufficiently large, then we have
\begin{equation}
\label{eqn:nkb}	
 \p[ \wt{N}(K,c) \leq b_1 K ] \le c_2(\wt a,b_1) e^{-\wt a K}.
 \end{equation}
Therefore
\begin{align*}
\p[ N(K,c) \leq b_1 K]
&= \p[N(K,c) \leq b_1 K, E(K,b)] + O(e^{-a K}) \quad\text{(by~\eqref{eq:EbK})}\\
&\leq \p[\wt{N}(K,c) \leq b_1 K] + O(e^{-a K})\\
&= O(e^{-a_1 K}) \quad\text{(by~\eqref{eqn:nkb})}
\end{align*}
where $a_1=\wt a\wedge a$. Since we can choose $\wt a$ and $a$ to be arbitrarily large, $a_1$ can also be arbitrarily large. Also note that we can choose $b$ arbitrarily close to $1$ and $b_1$ arbitrarily close to $b$.
\end{proof}

Finally, let us deduce the following upper bound for the Minkowski dimension of a geodesic using~\eqref{eqn:dist_across}.

\begin{proposition}
\label{prop:dimension_bound}
There exists $d \in [1,2)$ so that almost surely the upper Minkowski dimension of any $\qdist_h$-geodesic is at most $d$.
\end{proposition}

We will make use of Proposition~\ref{prop:dimension_bound} in the proof of Theorem~\ref{thm:removable} where it will be used to control the number of elements in a Whitney cube decomposition of a given size in the complement of a geodesic.

\begin{proof}[Proof of Proposition~\ref{prop:dimension_bound}]
Fix $z\in\C$ and $r>0$ and also consider the event $E(K,b)$. Fix  $a,b$ and $M$ so that~\eqref{eq:EbK} holds.  Let $S_{1,k}$ be the infimum of $\qdist_{h}$-lengths of paths contained in $B(z,7r_k/8) \setminus B(z,3r_k/4)$ which separate $0$ from $\infty$.  We also let $S_{2,k}$ be the $\qdist_h$ distance from $\partial B(z,5r_k/8)$ to $\partial B(z,r_k/2)$.  Let $g_k$ and $E(K,b)$ be as in the proof of Proposition~\ref{prop:dist_around}.  Let $F(K,b)$ be the event that $S_{1,g_k} \geq S_{2,g_k}$ for every $k\in\{1,\ldots, bK\}$ and let $F(K)$ be the event that $S_{1,k} \geq S_{2,k}$ for every $k \in \{1,\ldots,K\}$.  Then we have that
\begin{align*}
      \p[F(K)]
&= \p[F(K), E(K,b)] + O(e^{-aK}) \quad\text{(by~\eqref{eq:EbK})}\\
&\leq \p[F(K,b)] + O(e^{-aK})\\
&\leq (1-p)^{bK} + O(e^{-aK}) \quad\text{(by~\eqref{eqn:dist_across})}.
\end{align*}
Fix $\eps>0$ small and $K=\lceil\log_2\eps^{-1}\rceil$.  Then we have shown that $\p[F(K)] = O(\eps^\delta)$ where $\delta=\min(a\log_2 e, -b\log_2(1-p))>0$.

Fix $d\in[1,2)$.  We will set its precise value later in the proof.  For each pair of disjoint compact sets $H_1, H_2 \subseteq \C$ and compact set $A \subseteq \C$ containing $H_1,H_2$, we let $G(H_1,H_2; A)$ be the set of all $\qdist_h$-geodesics with one endpoint in $H_1$ and the other endpoint in $H_2$ and which are contained in~$A$.  We aim to prove that almost surely, every $\qdist_h$-geodesic in $G(H_1, H_2; A)$ has upper Minkowski dimension at most $d$.  As we can take $H_1,H_2,A$ to be squares centered at rational points with rational side lengths, the union of the sets $G(H_1, H_2; A)$ covers the set of all $\qdist_h$-geodesics. This will imply that the event that the upper Minkowski dimension of every $\qdist_h$-geodesic is at most $d$ has probability $1$, since we can write it as a countable intersection of events which all occur with probability~$1$.

Fix $a\in(0,1)$ and $\eps,r>0$ and assume that $r < \dist(H_1,H_2)/2$. Then we can cover $H_1$ and $H_2$ by balls of radius $r$ centered at points in $r \Z^2$. For every $x,y \in r \Z^2$ with $B(x,r)\cap H_1 \neq \emptyset$ and $B(y, r) \cap H_2 \neq \emptyset$, let $G(x,y, r; A)$ be the set of all $\qdist_h$-geodesics from $B(x, r)$ to $B(y, r)$ which are contained in~$A$ and let $U(x,y, r; A)$ be the union of all $\qdist_h$-geodesics in $G(x,y, r; A)$.  Fix $z \in A$ and $r>0$ such that $B(z,r) \cap (B(x, r)\cup B(y, r)) = \emptyset$.  In the notation of the first paragraph of the proof, if $S_{1,k} < S_{2,k}$ for some $k\in \{1,\ldots,K\}$, then it is impossible for any $\qdist_h$-geodesic with endpoints outside of $B(z,r)$ to hit $B(z,2^{-K}r)$, hence also $B(z, \eps r/2)$; see the left side of Figure~\ref{fig:shortcut}.  It then follows that for any $r > 0$ the number of balls of radius $\eps r/2$ that one needs to cover $U(x,y, r; A) \setminus (B(x, 2r) \cup B(y, 2r))$ is dominated from above by the number $N$ of $z \in ((\eps r/2) \Z^2) \cap A$ for which $F(K)$ holds.  We emphasize that this upper bound does not depend on $x$ or $y$.  Now fix $\zeta>0$ and take $r=\zeta^a$, $\eps=\zeta^{1-a}$.  Then $\E[N] = O(\zeta^{\delta(1-a)-2})$.  On the other hand, the number of balls of radius $\zeta/2$ that one needs to cover $U(x,y, \zeta^a; A) \cap  (B(x, 2\zeta^a) \cup B(y, 2\zeta^a))$ is at most $C_0 \zeta^{a-2}$ where $C_0 > 0$ is a constant which does not depend on $x,y$ or $\zeta$.  Let $d=\max(2-\delta(1-a), 2-a)$.  We have proved that the number of balls of radius $\zeta/2$ that one needs to cover $U(x,y,\zeta^a; A)$ is at most $N + C_0 \zeta^{-d}$.  Since this upper bound does not depend on $x,y$, it follows that every geodesic  in $G(H_1, H_2; A)$ can be covered by at most $N + C_0 \zeta^{-d}$ balls  of radius $\zeta/2$. 
Since this is true for all small $\zeta > 0$ and $\E[N] = O(\zeta^{\delta(1-a)-2})=O(\zeta^{-d})$, it follows that every geodesic in $G(H_1, H_2; A)$ has upper Minkowski dimension at most $d$.  This completes the proof.
\end{proof}

\begin{figure}[ht!]
\begin{center}
\includegraphics[width=0.9\textwidth]{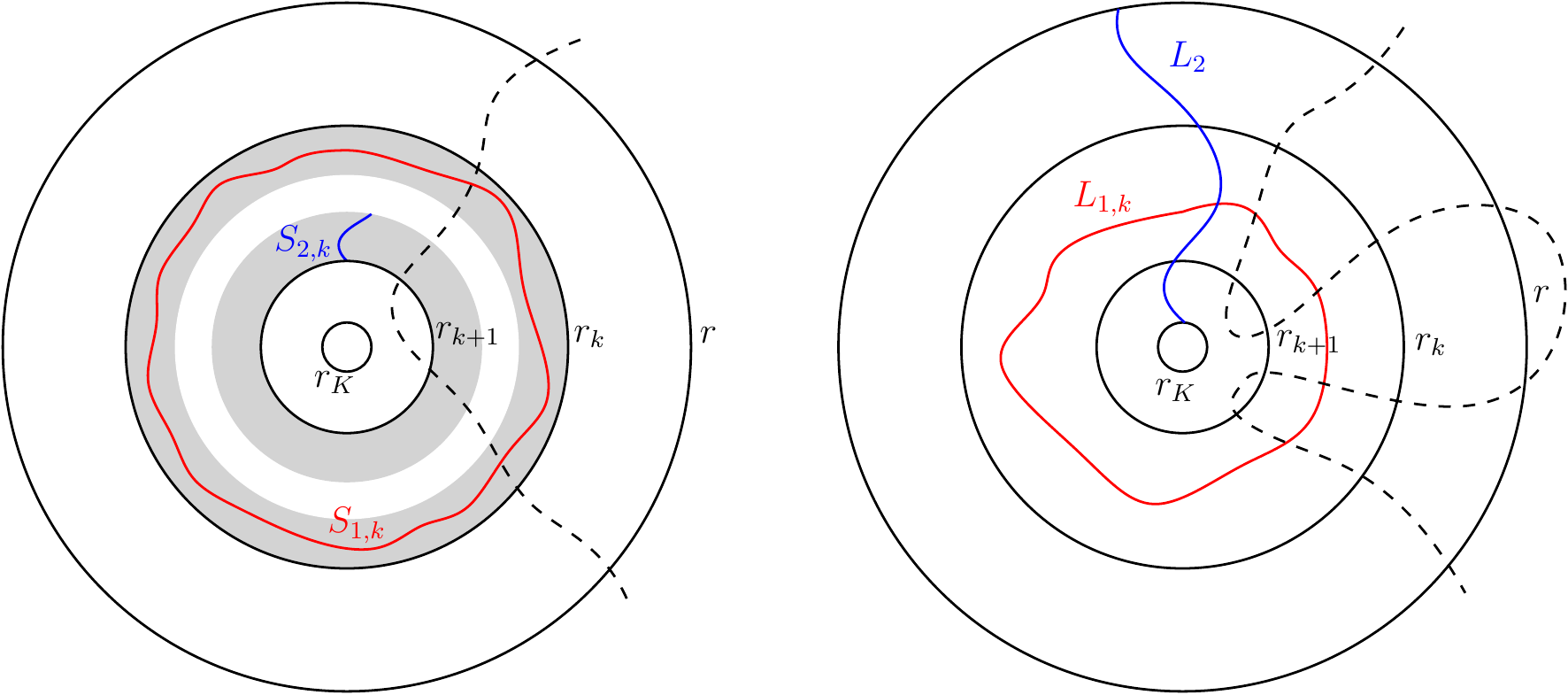}
\end{center}
\caption{\label{fig:shortcut} We draw the successive scales.  \textbf{Left:} We show in red (resp.\ blue) the path which realizes the minimal length $S_{1,k}$ (resp.\ $S_{2,k}$).  If for some $k\in \{1,\ldots,K\}$, one has $S_{1,k}< S_{2,k}$, then any geodesic with both endpoints outside of $B(z,r)$ cannot enter $B(z, r_{k+1})$.  \textbf{Right:} We show in red (resp.\ blue) the path which realizes the minimum length $ L_{1,k}$ (resp.\ $ L_{2}$). If $L_{1,k}<L_2$, then any geodesic cannot make more than four crossings across the annulus $B(z,r)\setminus \ol{B(z, r_{k+1})}$.  In both pictures, the dashed curves represent configurations of  geodesics which are impossible, since the red curves are shortcuts.}
\end{figure}

\subsection{Proof of Theorems~\ref{thm:singular} and~\ref{thm:regularity}}

\begin{proof}[Proof of Theorem~\ref{thm:singular}]
Fix $z\in \C, \eps>0, \zeta>1$. Let $L_2$ be the $\qdist_h$-distance from $\partial B(z,\epsilon^\zeta)$ to $\partial B(z,\epsilon)$.  Fix $K=\lfloor \log_2 \eps^{1-\zeta}\rfloor$. For $k\in[1,K]$, let $L_{1,k}$ and $L_{2,k}$ be as in Proposition~\ref{prop:dist_around} for $r_k = 2^{-k} \epsilon$.  See Figure~\ref{fig:shortcut} (right).
Note that
\[ L_2 \geq \sum_{k=1}^{K} L_{2,k}.\]
Consequently, the fraction $\rho$ of $k \in \{1,\ldots,K\}$ for which
\begin{equation}
\label{eqn:l2k_good}
L_{2,k} \leq \frac{c_1}{K} L_2
\end{equation}
is at least $1-1/c_1$.  We will chose $c_1 =100$ so that $\rho \geq 99/100$.

By Proposition~\ref{prop:dist_around}, for any $a>0$, we can choose a value of $c_2 > 0$ large so that the fraction of $k\in[1,K]$ with
\begin{equation}
\label{eqn:l1k_good}
L_{1,k} \leq c_2 L_{2,k}	
\end{equation}
is at least $99/100$ with probability $1-O(e^{-aK})=1-O(\eps^{a(\zeta-1)\log_2 e})$.  On this event, there must exist $k_0$ for which both~\eqref{eqn:l2k_good} and~\eqref{eqn:l1k_good} occur.  We then have that
\[ L_{1,k_0} \leq c_2 L_{2,k_0} \leq \frac{c_1 c_2}{K} L_2.\]
We emphasize that the values of $c_1,c_2$ do not depend on $\epsilon$.  Therefore by choosing $\epsilon > 0$ sufficiently small (hence $K$ is big), we have that $L_{1,k_0} < L_2$.  This implies that it is not possible for a geodesic to have more than four crossings across the annulus $B(z,\epsilon) \setminus \ol{B(z,\epsilon^\zeta)}$ because in this case we have exhibited a shortcut.  See the right side of Figure~\ref{fig:shortcut}. Therefore, the probability for a geodesic to have more than four crossings across the annulus $B(z,\epsilon) \setminus \ol{B(z,\epsilon^\zeta)}$ is at most $O(\eps^{a(\zeta-1)\log_2 e})$, where the exponent $a(\zeta-1)\log_2 e$ can be made arbitrarily large, since $a$ can be made arbitrarily big.  In particular, it implies that if $\eta$ is a geodesic from $0$ to any point outside of $B(0,2)$, then by the Borel-Cantelli lemma there a.s.\ exists $\eps_0>0$ so that for all $\eps\in(0,\eps_0)$ and all $z\in B(0,2) \setminus \ol{\D}$, $\eta$ does not make more than four crossings across the annulus $B(z,\epsilon) \setminus \ol{B(z,\epsilon^\zeta)}$.  However, this same event has probability zero for any whole-plane SLE$_\kappa$ curve (provided we choose $\zeta > 1$ sufficiently close to $1$ depending on $\kappa$), by Proposition~\ref{prop:whole_plane_sle_crossings}. Therefore, the law of the geodesic $\eta$ is singular w.r.t.\ the law of a whole-plane SLE curve. We have thus completed the proof.
\end{proof}

\begin{proof}[Proof of Theorem~\ref{thm:regularity}]
Fix $\delta \in (0,1)$ and $R > 0$. Let $\eta$ be any geodesic contained in $B(0,R)$. Since $R$ can be arbitrarily large,
it suffices to prove the result for $\eta$.  Let $\CN_k  = (2^{-k} \Z)^2 \cap B(0,R)$.  The proof of Theorem~\ref{thm:singular} implies that there a.s.\ exists $k_0 \in \N$ so that $k \geq k_0$ implies that the following is true.  The geodesic $\eta$ cannot make four crossings across the annulus $B(z,2^{(1-\delta)(1-k)}) \setminus \ol{B(z,2^{1-k})}$ for $z \in \CN_k$.

Fix times $0 < s < t$. 
If $|\eta(s) - \eta(t)|\ge 2^{-k_0}$, then we can choose $C(\delta,\eta)=\diam(\eta) \, 2^{k_0(1-\delta)}$  in~\eqref{eqn:regularity}. Otherwise, we can find $k\ge k_0$ so that $2^{-k-1} \leq |\eta(s) - \eta(t)| < 2^{-k}$.  Then we have that $\eta(s),\eta(t) \in B(z,2^{1-k})$ for some $z \in \CN_k$.  If $\eta([s,t])$ were not contained in $B(z,2^{-(1-\delta)k})$, then $\eta$ would make four crossings from $\partial B(z,2^{1-k})$ to $\partial B(z,2^{(1-\delta)(1-k)})$.  Therefore $\eta([s,t])$ is contained in $B(z,2^{-(1-\delta)k})$, which completes the proof.
\end{proof}

\section{Conformal removability}
\label{sec:conf_remov}

In this section, we aim to prove Theorem~\ref{thm:removable}, i.e., almost surely any geodesic $\eta$ is conformally removable.  We will rely on a sufficient condition by Jones and Smirnov \cite{js2000remove} to prove the removability of $\eta$, which we will now describe.  Let $\CW$ be a Whitney cube decomposition of $\C \setminus \eta$. Among other properties, $\CW$ is a collection of closed squares whose union is $\C\setminus\eta$ and whose interiors are pairwise disjoint. Moreover, if $Q \in \CW$ then $\dist(Q,\eta)$ is within a factor $8$ of the side-length $|Q|$ of $Q$.  Let $\varphi \colon \D \to \C \setminus \eta$ be the unique conformal transformation with $\varphi(0) = \infty$ and $\lim_{z\to 0}z\varphi(z)>0$.  We define the \emph{shadow} $s(Q)$ as follows (see Figure~\ref{fig:shadow}).  Let $I(Q)$ be the radial projection of $\varphi^{-1}(Q)$ onto $\partial \D$.  That is, $I(Q)$ consists of those points $e^{i \theta}$ for $\theta \in [0,2\pi)$ such that the line $r e^{i \theta}$, $r \in [0,1]$, has non-empty intersection with $\varphi^{-1}(Q)$.  We then take $s(Q) = \varphi(I(Q))$.

It is shown by Jones and Smirnov in \cite{js2000remove} that to prove that $\eta$ is conformally removable, it suffices to check that
\begin{equation}
\label{eqn:shadow_sum_bound}
\sum_{Q \in \CW} \diam( s(Q))^2 < \infty.
\end{equation}
This is the condition that we will check in order to prove Theorem~\ref{thm:removable}.
\begin{figure}[ht!]
\begin{center}
\includegraphics[width=0.9\textwidth]{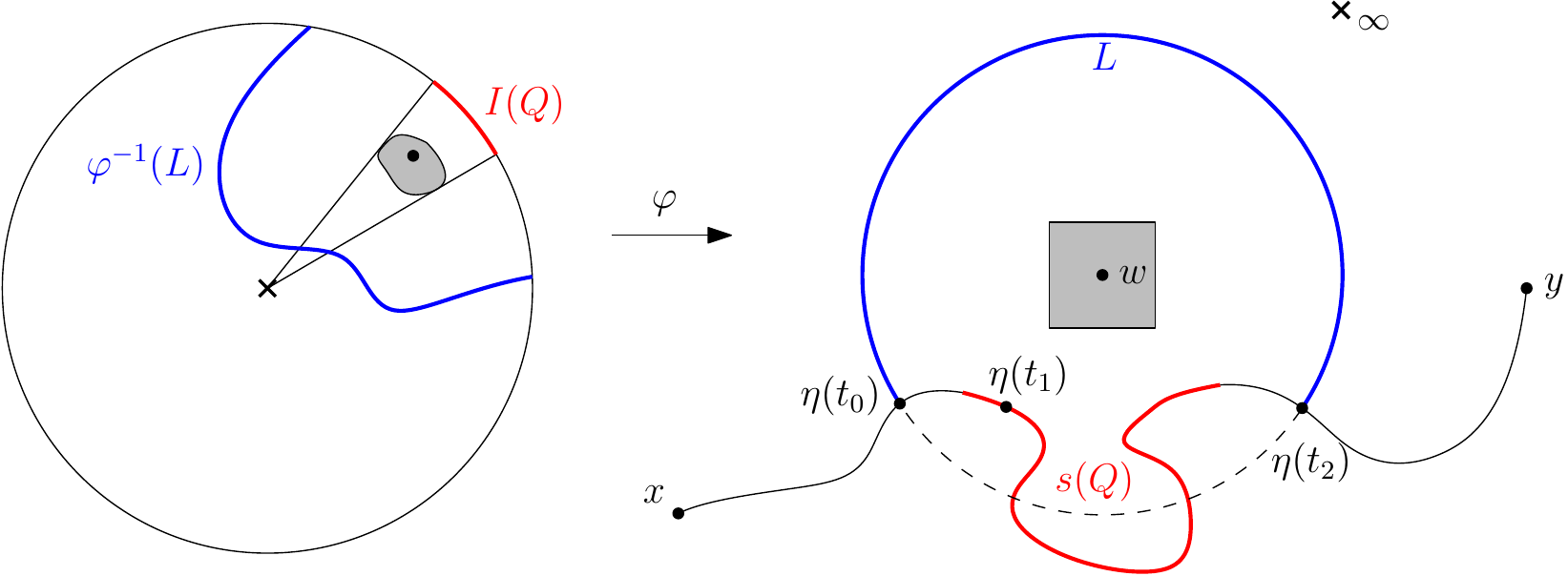}
\end{center}
\caption{\label{fig:shadow} We depict the conformal map $\varphi$ from $\D$ onto $\C\setminus\eta$, where $\eta$ is a geodesic from $x$ to $y$ shown on the right. On the right, we show one Whitney cube $Q$ centered at $w$ and its shadow $s(Q)$ in red. The blue arc $L$ is used in the proof of Lemma~\ref{lem:shadow_bound}. The pre-images of $Q$, $s(Q)$ and $L$ under $\varphi$ are  shown on the left.}
\end{figure}

\begin{lemma}
\label{lem:shadow_bound}
For each $\delta \in (0,1)$ there a.s.\ exists a constant $C(\delta,\eta) > 0$ so that the following is true.  For each $Q \in \CW$ with $|Q| = 2^{-n}$ we have that
\[ \diam(s(Q)) \leq C(\delta,\eta) 2^{-n(1-\delta)}.\]
\end{lemma}
\begin{proof}
Fix $Q \in \CW$ with $|Q| = 2^{-n}$.  By the definition of the Whitney cube decomposition, we have that $\dist(Q,\eta) \in [2^{-n-3}, 2^{-n+3}]$.  Let $w$ be the center of $Q$.  See Figure~\ref{fig:shadow} for illustration.

By Lemma~\ref{lem:points_distance}, for all $r\in(0,1)$ and all $z$ such that $|z-w|\le r \dist(w,\eta)$, we have
\begin{align*}
|\varphi^{-1}(z)-\varphi^{-1}(w)| \le \frac{4r}{1-r^2} \dist(\varphi^{-1}(w), \partial \D).
\end{align*}
This implies that $\varphi^{-1}(Q)$ is contained in a ball centered at $\varphi^{-1}(w)$ with radius at most a constant times $\dist(\varphi^{-1}(w), \partial \D).$ This implies that there exists $c_0>0$ such that
\begin{align}\label{eq:IQ}
\diam(I(Q)) \le c_0 \dist(\varphi^{-1}(w),\partial \D).
\end{align}

Let us parameterize $\eta$ continuously by $t\in[0,1]$ so that $\eta(0)=x$ and $\eta(1)=y$.
 Let $t_1$ be the first time $t$ that $t \mapsto \dist(\eta(t),Q)$ achieves its infimum.  We then let $t_0$ (resp.\ $t_2$) be the first (resp.\ last) time $t$ before (resp.\ after) $t_1$ that $\dist(\eta(t),w) = 2^{-n(1-\delta)}$.  Let $I = \eta([t_0,t_2])$.  By~\eqref{eqn:regularity}, there exists $ \wt C(\delta,\eta)>0$ such that
\[ \diam(I) \le \wt C(\delta,\eta) 2^{-n(1-\delta)^2} \le \wt C(\delta,\eta) 2^{-n(1-2\delta)}.\]
To complete the proof, it suffices to show that $s(Q) \subseteq I$.

Let $L$ be the connected component of $\partial B(w,2^{-n(1-\delta)})\setminus \eta$ which together with $\eta$ separates $w$ from $\infty$.  The Beurling estimate implies that the probability that a Brownian motion starting from $w$ exits  $\C \setminus (\eta \cup L)$ in $L$ is $O(2^{-n \delta/2})$.  By the conformal invariance of Brownian motion, we therefore have that the probability that a Brownian motion starting from $\varphi^{-1}(w)$ hits $\varphi^{-1}(L)$ before hitting $\partial \D$ is $O(2^{-n \delta/2})$.  If $\varphi^{-1}(L)$ had an endpoint in $I(Q)$, then due to~\eqref{eq:IQ}, this probability would be bounded from below. Therefore this cannot be the case, so $\varphi^{-1}(I)$ must contain $I(Q)$.  That is, $I$ contains $s(Q)$.
\end{proof}

\begin{proof}[Proof of Theorem~\ref{thm:removable}]
As we have mentioned above, it suffices to show that the sum~\eqref{eqn:shadow_sum_bound} is a.s.\ finite.   

Proposition~\ref{prop:dimension_bound} implies that there exists $d \in [1,2)$ and $n_0>0$ such that for all $n\ge n_0$, one can cover $\eta$ with a collection of $O(2^{nd})$ balls of radius $2^{-n}$. We denote by  $\CC_n$ the collection of the centers of these balls. 
For any $Q \in \CW$ with $|Q| = 2^{-n}$,  since $\dist(Q,\eta) \in [2^{-n-3}, 2^{-n+3}]$, $Q$ must be contained in $B(z, 2^{-n+4})$ for some $z\in\CC_n$.
Since all the cubes in $\CW$ are disjoint, a ball $B(z, 2^{-n+4})$ can contain at most $2^{10}$ cubes in $\CW$ of side length $2^{-n}$. This implies that the number of cubes in $\CW$ of side length $2^{-n}$ is $O(2^{nd})$.

On the other hand, Lemma~\ref{lem:shadow_bound} implies that the diameter of a shadow of a cube in $\CW$ with side length $2^{-n}$ is  $O(2^{-n(1-\delta)})$.
Therefore the total contribution to~\eqref{eqn:shadow_sum_bound} coming from cubes of side length $2^{-n}$ is  $O(2^{-2n(1-\delta)} \times 2^{dn})$.  We can take $\delta \in (0,1)$ small enough so that $d-2(1-\delta) < 0$ so that the sum over $n$ is finite.  This completes the proof.
\end{proof}

\appendix

\section{$\SLE$ almost surely crosses mesoscopic annuli}
\label{sec:SLE}

The purpose of this appendix is to prove Propositions~\ref{prop:whole_plane_sle_crossings} and~\ref{prop:sle_crossings}.  We will begin by proving a lower bound for the probability that chordal $\SLE_\kappa$ makes $k$ crossings across an annulus (Lemma~\ref{lem:crossing_lower_bound}) and then use this lower bound to complete the proof of Propositions~\ref{prop:whole_plane_sle_crossings} and~\ref{prop:sle_crossings}.  Throughout, we will assume that we have fixed $\kappa > 0$ and that $\eta$ is an $\SLE_\kappa$ in $\h$ from $0$ to $\infty$.

\begin{lemma}
\label{lem:crossing_lower_bound}
There exist constants $c_2,c_3 > 0$ depending only on $\kappa$ so that the following is true.  For each $z \in \D$ with $\im(z) \geq 1/50$ and $\epsilon \in (0,1/200)$, the probability that $\eta$ makes at least $k$ crossings from $\partial B(z,\epsilon)$ to $\partial B(z,1/100)$ before exiting $B(0,2)$ is at least $c_2 \epsilon^{c_3 k^2}$.
\end{lemma}

We believe that the exact exponent in the statement of Lemma~\ref{lem:crossing_lower_bound} should be equal to the interior arm exponent for SLE. This was computed in \cite{wu2018arms} but in a setup which we cannot use to prove Propositions~\ref{prop:whole_plane_sle_crossings} and~\ref{prop:sle_crossings}.  We will give an elementary and direct proof of Lemma~\ref{lem:crossing_lower_bound}.

Before we give the proof of Lemma~\ref{lem:crossing_lower_bound}, we will first recall the form of the SDE which describes the evolution in $t$ of $\pi$ times the harmonic measure of the left side of the outer boundary of $\eta([0,t])$ and $\R_-$ as seen from a fixed point in $\h$.  Let $U = \sqrt{\kappa} B$ be the Loewner driving function for $\eta$, fix $z \in \h$, and let
\[ Z_t(z) = X_t + iY_t = g_t(z) - U_t \quad\text{and}\quad \Theta_t = \arg Z_t.\]
Then $\Theta_t$ gives $\pi$ times the harmonic measure of the left side of the outer boundary of $\eta([0,t])$ and $\R_-$ as seen from $z$.  Let $\wh{\Theta}$ be given by $\Theta$ reparameterized according to $\log$ conformal radius as seen from $z$.  Then $\wh{\Theta}_t$ satisfies the SDE
\begin{equation}
\label{eqn:theta_sde}
d \wh{\Theta}_t = \left(1- \frac{4}{\kappa} \right) \cot \wh{\Theta}_t dt + d \wh{B}_t
\end{equation}
where $\wh{B}$ is a standard Brownian motion (see, for example, \cite[Section~6]{lv2012multifractal}).

\begin{proof}[Proof of Lemma~\ref{lem:crossing_lower_bound}]
Let $\varphi$ be the unique conformal transformation from $\h$ to the half-infinite cylinder $\cyl = \R_+ \times [0,2\pi]$ (with the top and bottom identified) which takes $z$ to $\infty$ and $0$ to $0$.  See Figure~\ref{fig:cylinder}.
Since $z \in \D$ and $\im(z) \geq 1/50$, we note that the distance between $0$ and $\varphi(\infty)$ in $\cyl$ is bounded from below.  We will consider $\wt{\eta} = \varphi(\eta)$ in place of $\eta$ and we will define an event for $\wt{\eta}$ which implies that~$\eta$ makes at least~$k$ crossings from $\partial B(z,\epsilon)$ to $\partial B(z,1/100)$ before exiting $B(0,2)$.  
We can choose a universal constant $c_0 > 0$ large enough such that the following holds simultaneously for all $z \in \D$ with $\im(z) \geq 1/50$:
\begin{equation}
\label{eqn:c_0_choice}
[\log \epsilon^{-1} + c_0,\infty) \times [0,2\pi] \subseteq \varphi(B(z,\epsilon)) \quad\text{and}\quad  [0,\tfrac{1}{c_0}] \times [0,2\pi] \subseteq \varphi(\h \setminus B(z,\tfrac{1}{100})).	
\end{equation}
We then define a deterministic path $\Gamma$ as follows.
For $0 \leq j \leq k$, let
\begin{align*}
 x_{4j} &= \frac{1}{c_0} \cdot \one_{j \geq 1} + i\frac{2j}{k}, \quad x_{4j+1} = \log \epsilon^{-1} + c_0 + i \frac{2j}{k},\\
 x_{4j+2} &= \log \epsilon^{-1} + c_0 + i \frac{2j+1}{k},\quad x_{4j+3} = \frac{1}{c_0} + i \frac{2j+1}{k}.
\end{align*}
Let $\Gamma$ be the path which visits the points $x_0,\ldots,x_{4k}$ in order by:
\begin{itemize}
\item traveling from $x_{4j}$ to $x_{4j+1}$ linearly to the right,
\item from $x_{4j+1}$ to $x_{4j+2}$ counterclockwise along an arc connecting $x_{4j+1}$ and $x_{4j+2}$,
\item from $x_{4j+2}$ to $x_{4j+3}$ linearly to the left, and
\item from $x_{4j+3}$ to $x_{4j+4}$ clockwise along an arc connecting $x_{4j+3}$ and $x_{4j+4}$.
\end{itemize}
We parameterize $\Gamma$ at unit speed and we choose the arcs in the definition of $\Gamma$ so that it is a $C^2$ curve.  In particular, we can arrange so that the second derivative of $\Gamma$ is $O(k)$.
\begin{figure}[h]
\begin{center}
\includegraphics[width=.9\textwidth]{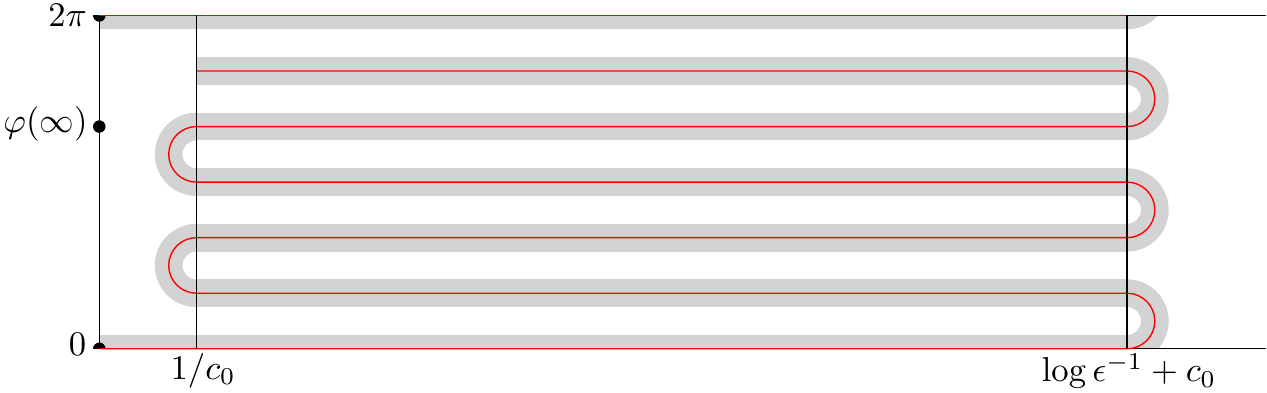}
\end{center}
\caption{\label{fig:cylinder} 
The cylinder $\cyl = \R_+ \times [0,2\pi]$ and the path $\Gamma$ (in red). We will show that $\wt\eta$ stays in the $c_1^{-3/2} (2k)^{-1}$-neighborhood of $\Gamma$ (in grey) with probability at least $c_2 \eps^{c_3k^2}$ for some $c_2, c_3>0$.}
 
\end{figure}
The rest of the proof will be dedicated to proving that the following event holds with probability at least $c_2 \eps^{c_3 k^2}$ for some $c_2, c_3>0$:
\begin{align}\label{event:tube}
\wt{\eta} \text{ reaches distance } (2c_0 k)^{-1} \text{ of } x_{4k} \text{ before leaving the } (2c_0 k)^{-1}\text{-neighborhood of } \Gamma.  
\end{align}
Note that this will complete the proof, since the event~\eqref{event:tube} implies that $\eta$ makes at least $k$ crossings from $\partial B(z,\epsilon)$ to $\partial B(z,1/100)$ before exiting $B(0,2)$.

Recall that we have parameterized $\Gamma$ at unit speed.  Let $[0,T]$ be the time interval on which it is defined and note that $T \asymp k \log \epsilon^{-1}$.  
Let $0 = t_0 < t_1 < \cdots < t_n = T$ be equally spaced times with $n = \lfloor c_1 k^2 \log \epsilon^{-1}\rfloor$ where $c_1 > 0$ is a large constant we will adjust later.  
For each $1 \leq j \leq n$, we let $y_j = \Gamma(t_j)$. Note that the spacing between the $y_j$ is of order $c_1^{-1} k^{-1}$. 
 Let $D_j$ be the sector formed by the two infinite lines with slopes $c_1^{-19/64}$ and $-c_1^{-19/64}$ relative to the tangent of $\Gamma$ at $\Gamma((t_{j-1}+t_j)/2)$ (see Figure~\ref{fig:line}).  Let $\tau_j = \inf\{t \geq \tau_{j-1} : \wt{\eta}(t) \in \partial D_j \}$. Let $\ol{\Theta}_t^j$ be the harmonic measure of the left side of the outer boundary of $\wt{\eta}([0,t])$ and $\varphi(\R_-)$ as seen from $y_j$.   We inductively define events $E_j$ as follows.  Let $E_0$ be the whole sample space.  Given that $E_0,\ldots,E_{j}$ have been defined, we let $E_{j+1}$ be the event that $ E_{j}$ occurs, $\tau_{j+1}< \infty$, and 
\begin{itemize}
\item $\ol{\Theta}_t^{j+1}|_{[\tau_{j},\tau_{j+1}]}$ differs from $\tfrac{1}{2}$ by at most $c_1^{-17/64}$ and
\item $\ol{\Theta}^{j+1}_{\tau_{j+1}}$ differs from $\tfrac{1}{2}$ by at most $c_1^{-19/64}$.
\end{itemize}

\begin{figure}[h]
\begin{center}
\includegraphics[width=\textwidth]{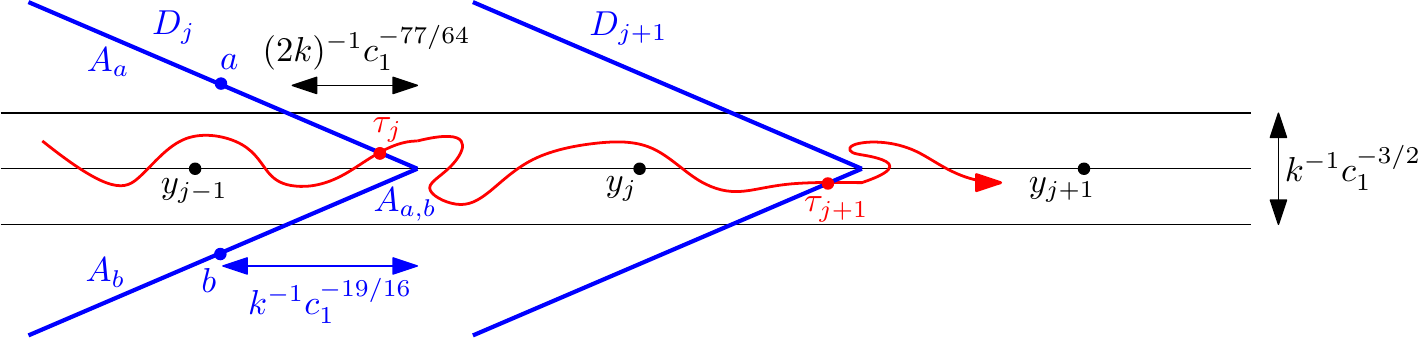}
\end{center}
\caption{\label{fig:line} Illustration of the definitions of the points $y_j = \Gamma(t_j)$, the sectors $D_j$, the stopping times $\tau_j$ and the sets  $A_a$, $A_b$, and $A_{a,b}$.} 
\end{figure}

Let us first prove by induction that the following statement is true for all $1\le j\le n$:
\begin{enumerate}
\item[(I$_j$)] On the event $E_j$, $\wt{\eta}([0,\tau_j])$ is contained in the  $c_1^{-3/2}(2k)^{-1}$-neighborhood of  $\Gamma.$
\end{enumerate}
Note that (I$_0$) is obviously true. Suppose that (I$_j$) holds, let us prove that (I$_{j+1}$) also holds. It suffices to prove that $\wt\eta([\tau_j, \tau_{j+1}])$ is contained in the  $c_1^{-3/2}(2k)^{-1}$-neighborhood of  $\Gamma.$ 
Suppose that it is not the case, so there exists $t\in(\tau_j, \tau_{j+1}]$ such that the distance between $\eta(t)$ and $\Gamma$ is equal to $c_1^{-3/2}(2k)^{-1}$. Then the harmonic measure of the left side of $\wt{\eta}([0,t])$ and $\varphi(\R_-)$ as viewed from  $y_{j+1}$ would differ from $\tfrac{1}{2}$ by at least a constant times $c_1^{-{1/4}}$ (which comes from $(c_1^{-3/2}k^{-1}/(c_1^{-1}k^{-1}))^{1/2}$), 
which is impossible since we are on $E_{j+1}$. 
This completes the induction step, hence (I$_j$) is true for all $1\le j\le n$.

As a consequence, it suffices to show that $\p[E_n] \ge c_2 \eps^{c_3 k^2}$ for some constants  $c_2, c_3>0$, in order to complete the proof of the lemma.

Let us first prove the following fact for all $1\le j\le n-1$:
\begin{equation}
\label{eqn:ej_thetaj_est}
\text{On the event } E_{j}, \quad \ol{\Theta}_{\tau_{j}}^{j+1} \quad\text{differs from}\quad \frac{1}{2} \quad \text{by at most} \quad c_1^{-9/32+o(1)}.
\end{equation}
Let $B^1$ (resp.\ $B^2$) be a Brownian motion started at $y_j$ (resp.\ $y_{j+1}$) and stopped upon hitting $\wt\eta([0,\tau_j])$. Let $T^1$ (resp.\ $T^2$) be the first time that $B^1$ (resp.\ $B^2$) hits $\partial D_j$. 
We will work on the event that $B^1$ (resp.\ $B^2$) stops in $D_j \cap B(y_j, k^{-1} c_1^{-19/64})$ (resp.\ $D_j \cap B(y_{j+1}, k^{-1} c_1^{-19/64})$) which happens with probability $1-O(c_1^{-45/128})$ by the Beurling estimate (since $c_1^{-45/128}<c_1^{-9/32}$, we can restrict ourselves on this event provided we have chosen $c_1 > 0$ large enough). On this event, we can think as if $\wt{\eta}([0,\tau_j])$ is contained in $D_j$.  Indeed, $\wt\eta([\tau_{j-1}, \tau_j])$ is by definition contained in $D_j$ and $D_j \cap B(y_j,k^{-1} c_1^{-19/64})$ contains the tube of width $(2k)^{-1} c_1^{-3/2}$ around $\Gamma([0,\tau_{j-1}]) \cap B(y_j,k^{-1} c_1^{-19/64})$ provided we choose $c_1 > 0$ large enough (recall that $\Gamma$ is a $C^2$ curve with $O(k)$ second derivative, so it differs at distance $x$ from the linear approximation corresponding to the tangent line by $O(k x^2)$ and this error term is at most a constant times $c_1^{-19/64} x$ for $x \leq k^{-1} c_1^{-19/64}$ provided we choose $c_1 > 0$ large enough).  Let $a$ and $b$ be points respectively on the upper and lower boundary of $D_j$ such that the distances between $a,b$ to $\Gamma((t_{j-1}+t_j)/2)$ are $k^{-1} c_1^{-19/16}$.  The points $a,b$ divide $\partial D_j$ into $3$ parts: one finite part that we denote by $A_{a,b}$ and two infinite half-lines with endpoints $a$ and $b$ that we denote by $A_a$ and $A_b$. See Figure~\ref{fig:line}.

Let $f_1$ (resp.\ $f_2$) be the conformal map from $\C\setminus D_j$ onto $\H$ which sends  $\infty$ to $\infty$, the tip of $D_j$ (i.e., $\Gamma((t_{j-1}+t_j)/2)$) to $0$, and such that  $\im(y_j)=1$ (resp.\ $\im(y_{j+1})=1$). For $\ell=1,2$, $f_\ell$ is a map of the form $w \mapsto a_\ell (w-b_\ell)^q$ where $q > 1/2$ and $a_\ell, b_\ell \in \C$ are such that $|a_\ell| \asymp (kc_1)^{q}$. The exponent $q \to 1/2$ as $c_1 \to \infty$, since the slope of $\partial D_j$  tends to $0$.
Therefore for $\ell=1,2$, the length of $f_\ell(A_{a,b})$ is 
\[ c_1^{1/2+o(1)}  (c_1^{-19/16+o(1)})^{1/2} = c_1^{-3/32+o(1)} \quad\text{as}\quad c_1 \to \infty.\]
On the other hand, since the curve $\Gamma$ is $C^2$ with $O(k)$ second derivative, the distance between $y_{j+\ell-1}$ and $f_\ell^{-1}(i)$ is $O(k) O(k^{-2} c_1^{-2}) = O(k^{-1} c_1^{-2})$. Noting that the derivative of $f_\ell$ at $f_\ell^{-1}(i)$ is $O(k c_1)$, we have that $\re(f_\ell(y_{j+\ell-1}))=\re(f_\ell(y_{j+\ell-1})-i)=O(k c_1) O(k^{-1}c_1^{-2}) =O(c_1^{-1})$ which is less than $c_1^{-3/32+o(1)}$. This implies that the harmonic measure of $A_{a,b}$ seen from $y_{j+\ell-1}$ is $c_1^{-3/32+o(1)}$ for $c_1$ large enough.

Note that we have the following facts for  $B^i$ for $i=1,2$:
\begin{itemize}
\item The event that $B^i(T^i)\in A_a \cup A_b$ has probability $1-c_1^{-3/32+o(1)}$. Conditionally on this event, the probability that  $B^i$  stops on the same side of  $\wt\eta([0,\tau_j])$ as $B^i(T^i)$ is $1-O(c_1^{-19/64})$.  Indeed, on $E_j$, by (I$_j$) we know that $\wt\eta(\tau_j)$ is in the $c_1^{-3/2} (2k)^{-1}$-neighborhood of $\Gamma$, hence has distance at most $k^{-1} c_1^{-77/64}/2$ to $\Gamma((t_{j-1}+t_j)/2)$. 
We condition on the point $B^i(T^i)$ and let $d_i$ denote the distance between $B^i(T^i)$ and $\Gamma((t_{j-1}+t_j)/2)$.
 Note that $d_i\ge k^{-1} c_1^{-19/16}$. 
 Since the slope of the lines which make the two sides of $\partial D_j$ is $c_1^{-19/64}$, $B^i(T^i)$ is at distance at most $2d_i c_1^{-19/64}$ to $\wt\eta([0,\tau_j])$. 
In order for $B^i$ to stop at the other side of $\wt\eta([0,\tau_j])$, it has to travel distance  at least $d_i- (2k)^{-1} c_1^{-77/64}$ before hitting $\wt{\eta}([0,\tau_j])$.
Consequently, conditionally on $B^i(T^i)$, the probability that  $B^i$  stops on the other side of $\wt\eta([0,\tau_j])$ as $B^i(T^i)$ is $O(d_i c_1^{-19/64} / (d_i- (2k)^{-1} c_1^{-77/64}))=O(c_1^{-19/64})$.
\item The event that $B^i(T^i)\in A_{a,b}$  has probability $c_1^{-3/32 + o(1)}$ as $c_1 \to \infty$. 
\end{itemize}
Recall that on the event $E_j$, the probability that $B^1$ stops on the left side of  $\wt\eta([0,\tau_j])$ (we denote this event by $B^1_\text{left}$) differs from $1/2$ by at most $O(c_1^{-19/64})$.
On the other hand, $\p[B^1_\text{left}]$ is also equal to
\begin{align*}
&\p[B^1(T^1)\in A_a] \p[B^1_\text{left} \mid B^1(T^1)\in A_a]  + \p[B^1(T^1)\in A_b] \p[B^1_\text{left} \mid B^1(T^1)\in A_b]  \\
&+ \p[B^1(T^1)\in A_{a,b}]  \p[B^1_\text{left} \mid B^1(T^1)\in A_{a,b}]\\
=&\p[B^1(T^1)\in A_a] (1-O(c_1^{-19/64})) + \p[B^1(T^1)\in A_b]  O(c_1^{-19/64}) \\
&+\p[B^1(T^1)\in A_{a,b}]  \p[B^1_\text{left} \mid B^1(T^1)\in A_{a,b}]\\
=&\p[B^1(T^1)\in A_a \cup A_b]/2 +O(c_1^{-19/64})+  \p[B^1(T^1)\in A_{a,b}]  \p[B^1_\text{left} \mid B^1(T^1)\in A_{a,b}]\\
=&1/2+O(c_1^{-19/64})+  \p[B^1(T^1)\in A_{a,b}] \!\left( \p[B^1_\text{left} \mid B^1(T^1)\in A_{a,b}]-1/2 \right)\\
=&1/2+O(c_1^{-19/64})+ c_1^{-3/32+o(1)}  \!\left( \p[B^1_\text{left} \mid B^1(T^1)\in A_{a,b}]-1/2 \right).
\end{align*}
Since the above should be equal to $1/2+ O(c_1^{-19/64})$, we must have
\begin{align}\label{eq:B1left}
\p[B^1_\text{left} \mid B^1(T^1) \in A_{a,b}]-1/2=O(c_1^{-(19/64-3/32+o(1))})\le c_1^{-13/64+o(1)}.
 \end{align}

We can further express $ \p[B^1_\text{left}\mid B^1(T^1)\in A_{a,b}]$ as an integration w.r.t.\ the position of $B^1(T^1)$ on $A_{a,b}$.  Note that conditionally on the event that $B^1(T^1)$ hits $A_{a,b}$, the point $f_1(B^1(T^1))$ is distributed according to a measure on $f_1(A_{a,b})$ which has Radon-Nikodym derivative at least $1-c_1^{-3/16+o(1)}$ w.r.t.\ the uniform measure on $f_1(A_{a,b})$. (Indeed, $f_\ell(y_{j+\ell-1})=i+ O(c_1^{-1})$ and the density at $x\in\R$ of the harmonic measure in $\H$ seen from $i$  is a constant times $1/(1+x^2)=1+O(x^2)$. Moreover, every $x\in A_{a,b}$ satisfies $|x| \le c_1^{-3/32+o(1)}$ as $c_1 \to \infty$.)  The same is true for $B^2$ and $T^2$ and  $f_2(A_{a,b})$. Note that the image under $f_2\circ f_1^{-1}$ of the uniform measure on  $f_1(A_{a,b})$ is equal to the uniform measure on $f_2(A_{a,b})$, since $f_1= c f_2$ for some $c>0$.  This implies that $ \p[B^2_\text{left} \mid B^2(T^2)\in A_{a,b}]$ differs from $ \p[B^1_\text{left} \mid B^1(T^1)\in A_{a,b}]$ by at most $c_1^{-3/16+o(1)}$, 
hence by~\eqref{eq:B1left} it also differs from $1/2$ by at most $c_1^{-3/16+o(1)}$. This implies that $ \p[B^2_\text{left}, B^2(T^2)\in A_{a,b}]$ differs from $ \p[B^2(T^2)\in A_{a,b}]/2$ by at most $c_1^{-3/32+o(1)}c_1^{-3/16+o(1)}= c_1^{-9/32+o(1)}$.  On the other hand,  we know that $\p[B^2_\text{left}, B^2(T^2)\in A_a \cup A_b]$ differs from $\p[B^2(T^2)\in A_a \cup A_b]/2$ by $O(c_1^{-19/64})$ which is smaller than $c_1^{-9/32+o(1)}$.  Hence~\eqref{eqn:ej_thetaj_est} is true.

Recall that  $\pi\ol{\Theta}_{t}^{j+1}$ evolves according to~\eqref{eqn:theta_sde} and its drift term tends to $0$ as  $\ol{\Theta}_{t}^{j+1}$ tends to $1/2$.  By~\eqref{eqn:ej_thetaj_est}, at time $\tau_j$, $\ol{\Theta}^{j+1}$ is in a $c_1^{-9/32+o(1)}$-neighborhood of $1/2$, hence it has a positive probability $p_0$ to remain in the (larger) $O(c_1^{-17/64})$-neighborhood of $1/2$ for $t\in [\tau_j, \tau_{j+1})$ and then stop in the $O(c_1^{-19/64})$-neighborhood of $1/2$ at $t=\tau_{j+1}$.

Let $\wt\CF_t:=\sigma(\wt\eta|_{[0,t]})$. It follows that for all $1\le j\le n-1$, we have
\begin{align*}
\p[E_{j+1} \giv \wt\CF_{\tau_j}] \mathbf{1}_{E_j} \ge p_0 \mathbf{1}_{E_j}.
\end{align*}
This implies that $\p[E_n] \ge p_0^n$. Since $n=c_1 k^2 \log \eps^{-1}$, this completes the proof.
\end{proof}

We will prove Proposition~\ref{prop:sle_crossings} by iteratively applying Lemma~\ref{lem:crossing_lower_bound} as $\eta$ travels from $0$ to $\partial \D$.
Let $m_1, m_2>0$ be constants that we will adjust later. For any $\eps>0$ and $j\in\N$, we define the stopping times $$\sigma_j = \inf\{t \geq 0 : \eta(t) \in \partial B(0, (m_1+m_2) j \epsilon)\}.$$ 
Let us first prove the following lemma.

\begin{lemma}
\label{lem:angle_good}
Fix $C > 0$.  Let $n(\eps)=((m_1+m_2) \epsilon)^{-1}$. There exist constants $\epsilon_0,c_1,c_2 > 0$ and $q_0 \in (0,1)$ so that for all $\epsilon \in (0,\epsilon_0)$, we have
\begin{equation}
\label{eqn:good_scales_happen}
 \p[ \im( \eta(\sigma_j) ) \leq C \epsilon \quad \text{for more than a $q_0$ fraction of} \quad 1 \leq j \leq n(\eps)] \leq c_1 e^{-c_2/\epsilon}.
\end{equation}
\end{lemma}
\begin{proof}
Let $\CF_t = \sigma(\eta(s) : s \leq t)$.  We will establish~\eqref{eqn:good_scales_happen} by showing that there exists a constant $p_0 > 0$ so that
\begin{equation}
\label{eqn:sle_go_up}
\p[ \im(\eta(\sigma_{j+1})) \geq C \epsilon \giv \CF_{\sigma_j} ] \geq p_0 \quad\text{for each}\quad j.
\end{equation}
Indeed,~\eqref{eqn:sle_go_up} implies that the number of $1 \leq j \leq n(\eps)$ for which $\im( \eta(\sigma_j) ) \geq C \epsilon$ is stochastically dominated from below by a binomial random variable with parameters $p=p_0$ and $n(\eps)$.  
Thus~\eqref{eqn:good_scales_happen} with $q_0 = 1-p_0$ follows from Lemma~\ref{lem:binomial}.

To see that~\eqref{eqn:sle_go_up} holds, fix a value of $j \in \N$ and let $\theta_j = \arg(\eta(\sigma_j))$.  Let $\ul{\theta}_j$ (resp.\ $\ol{\theta}_j$) be such that $[\ul{\theta}_j,\ol{\theta}_j]$ is the set of $\theta \in [0,\pi]$ so that the imaginary part of $(j+1)e^{i \theta}$ is at least $2 C \epsilon$.  We then let $z_j$ be the point on $\partial B(0,(m_1+m_2)(j+1)\epsilon)$ with argument $(\theta_j \vee \ul{\theta}_j) \wedge \ol{\theta}_j$.  We note that the harmonic measure as seen from $z_j$ of the part of $\partial \h_{\sigma_j}$ which is to the left (resp.\ right) of $\eta(\sigma_j)$ is at least some constant $a_0 > 0$.  Moreover, if $\im(\eta(\sigma_{j+1})) \leq C \epsilon$, then the harmonic measure seen from $z_j$ of either the part of $\partial \h_{\sigma_{j+1}}$ which is to the left or right of $\eta(\sigma_{j+1})$ will be at most some constant $a_1 > 0 $.  We note that from the explicit form of~\eqref{eqn:theta_sde} that there is a positive chance that $\Theta$ (with $w = z_j$) in the time interval $[\sigma_j,\sigma_{j+1}]$ starting from a point $(a_0,1-a_0)$ ends in $(a_1,1-a_1)$. On this event, $\im(\eta(\sigma_{j+1})) \geq C \epsilon$, which completes the proof of~\eqref{eqn:sle_go_up}.
\end{proof}

We let $(\sigma_{j_k})$ be the subsequence of $(\sigma_j)$ so that $\im(\eta(\sigma_j)) \geq C \epsilon$. 
For each $k$, let $z_k \in \partial B(0, ((m_1+m_2)j_k+m_1)\eps)$ be the point with the same argument as $\eta(\sigma_{j_k})$.  
Let $\phi_k$ be the unique conformal transformation $\h_{\sigma_{j_k}} \to \h$ which sends $\eta(\sigma_{j_k})$ to  $0$, $\infty$ to $\infty$ and such that $\im(\phi_k(z_k))=1/10$.   See Figure~\ref{fig:sle}.

\begin{figure}[ht!]
\begin{center}
\includegraphics[width=\textwidth]{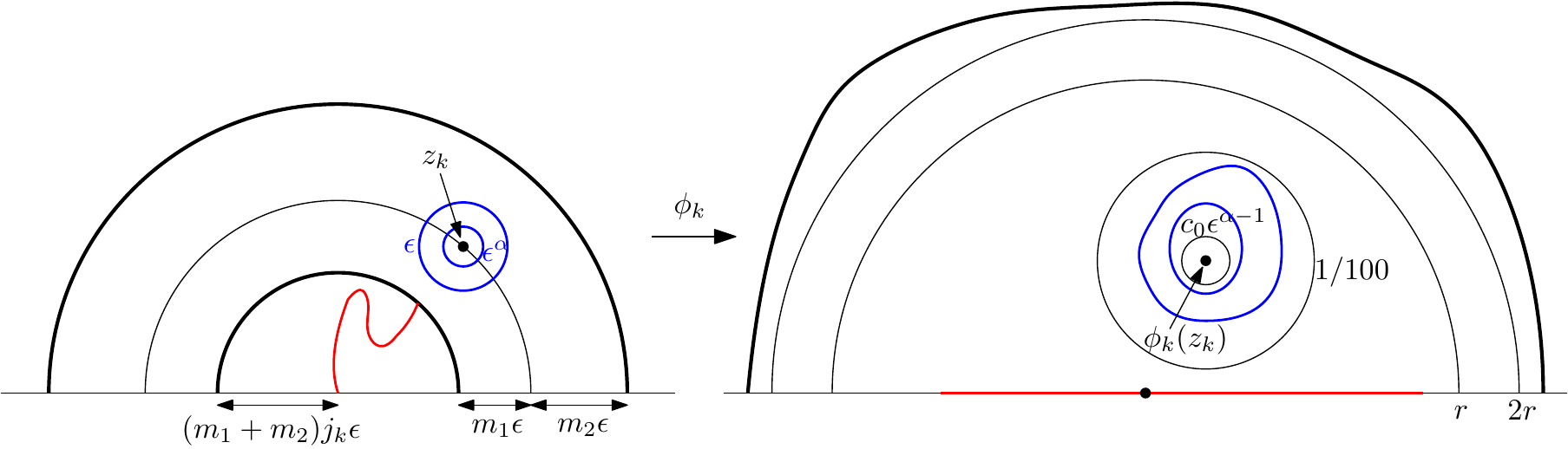}
\end{center}
\caption{\label{fig:sle} Illustration of the setup for the proof of Proposition~\ref{prop:sle_crossings}}
\end{figure}

\begin{lemma}
\label{lem:phi_deriv_est}
Fix $\alpha > 1$.  We can choose $m_1=C$ big enough so that there exists $c_0>0$ such that whenever $\eps$ is small enough, for all $k\in\N$, we have
\begin{align}\label{eq1}
B(\phi_k(z_k), c_0 \eps^{\alpha-1}) \subseteq \phi_k(B(z_k, \eps^\alpha))\subseteq \phi_k(B(z_k, \eps)) \subseteq B(\phi_k(z_k), 1/100).
\end{align}
With this value of $m_1=C$ chosen, there exists $r>0$ such that for all $k\in\N$, we have 
 \begin{align}\label{eq2}
 \phi_k(z_k)\in B(0,r) \cap\h.
\end{align}
We can finally choose $m_2$ big enough so that whenever $\eps$ is small enough, for all $k\in \N$,
\begin{align}\label{eq3}
B(0, 2r)\cap\h \subseteq \phi_k(B(0, (m_1+m_2)(j_k+1)\eps)\cap\h).
\end{align}
\end{lemma}

\begin{proof}

Let us first prove~\eqref{eq1}.
Lemma~\ref{lem:deriv_distance} implies that $|\phi_k'(z_k)|$ is within a factor of $4$ of $\dist(\phi_k(z_k),\partial \h)/\dist(z_k,\partial \h_{\sigma_k})$.  
By definition, $\dist(\phi_k(z_k),\partial \h)=1/10$. On the other hand, if we choose $C=m_1$, we have $\dist(z_k,\partial \h_{\sigma_k})=m_1\eps$.
It follows that $|\phi_k'(z_k)| \in (4^{-1}\cdot 10^{-1} m_1^{-1}\epsilon^{-1},  4\cdot 10^{-1} m_1^{-1}\epsilon^{-1})$.
By the Koebe $1/4$ theorem (Lemma~\ref{lem:koebe_quarter}), this implies $B(\phi_k(z_k), c_0 \eps^{\alpha-1}) \subseteq \phi_k(B(z_k, \eps^\alpha))$ for $c_0=m_1^{-1}/40$ and $\phi_k(B(z_k, \eps)) \subseteq B(\phi_k(z_k), r_0)$ for $r_0=4m_1^{-1}/10$. We can choose $m_1\ge 40$ so that $r_0\le 1/100$. This completes the proof of~\eqref{eq1}.

Let us then prove~\eqref{eq2}. For a Brownian motion started at $z_k$ and stopped upon exiting $\h_{\sigma_{j_k}}$, the probability that it hits the right hand-side of $\eta[0,\sigma_{j_k}]$ or $\R_+$ (resp.\ the left-hand side of $\eta[0,\sigma_{j_k}]$ or $\R_-$) is bounded below by some constant $c>0$. Since we have imposed $\im(\phi_k(z_k))=1/10$, it follows that there exists $r>0$ such that $|\re(\phi_k(z_k))|\le r$, because otherwise the harmonic measure seen from $\phi_k(z_k)$ of either $\R_-$ or $\R_+$ will be less than $c$. This completes the proof of~\eqref{eq2}.

Finally let us prove~\eqref{eq3}. For any $\delta>0$, we can choose $m_2$ big enough (with $m_1$ fixed) so that in $B(0,(m_1+m_2)(j_{k}+1)\eps)\cap \h_{\sigma_{j_k}}$, the harmonic measure seen from $z_k$ of $\partial B(0,(m_1+m_2)(j_{k}+1)\eps)\cap\h$  is at most $\delta$.
After applying the conformal map $\phi_k$, we have that the harmonic measure seen from $\phi_k(z_k)$ of $\phi_k \!\left(\partial B(0,(m_1+m_2)(j_{k}+1)\eps)  \cap\h\right)$  is at most $\delta$.
By choosing $\delta$ small enough, we can force $\partial B(0,(m_1+m_2)(j_{k}+1)\eps)$ to stay out of $B(0,2r)$.
This completes the proof of~\eqref{eq3}.
\end{proof}

\begin{proof}[Proof of Proposition~\ref{prop:sle_crossings}]
Fix $\alpha > 1$.  We will adjust its value later in the proof.  By Lemma~\ref{lem:crossing_lower_bound} and Lemma~\ref{lem:phi_deriv_est}, the conditional probability given $\CF_{\sigma_{j_k}}$ that $\eta$ makes $n$ crossings across $B(z_k,\eps) \setminus B(z_k,\eps^\alpha)$ before exiting $B(0, (m_1+m_2)(j_k+1)\eps)$ is at least $c_1 \eps^{c_2 n^2 (\alpha-1)}$ for constants $c_1,c_2 > 0$.  Since this is true for all $k$, by combining with Lemma~\ref{lem:angle_good} we see that the probability that $\eta$ fails to make $n$ such crossings for all $k$ with $\sigma_{j_k}$ before $\eta$ first hits $\partial \D$ is at most $(1-c_1 \eps^{c_2 n^2 (\alpha-1)})^{1/ (q_0 \eps)}$.  This tends to $0$ as $\epsilon \to 0$ provided we take $\alpha > 1$ sufficiently close to $1$, which completes the proof.
\end{proof}

\begin{proof}[Proof of Proposition~\ref{prop:whole_plane_sle_crossings}]
	This follows from Proposition~\ref{prop:sle_crossings} and the local absolute continuity between whole-plane and chordal $\SLE_\kappa$ \cite{sw2005coordinate}.
\end{proof}

\bibliographystyle{abbrv}
\bibliography{lqg_geodesics}

\begin{thebibliography}{10}

\bibitem{diabconf}
Open problems in quantum gravity at les diablerets.
\newblock
  \url{https://www.unige.ch/~smirnov/conferences/rpg09/Diablerets_open.pdf},
  2009.

\bibitem{bmr2016classification}
E.~{Baur}, G.~{Miermont}, and G.~{Ray}.
\newblock {Classification of scaling limits of uniform quadrangulations with a
  boundary}.
\newblock {\em ArXiv e-prints}, Aug. 2016.

\bibitem{b2010icm}
I.~Benjamini.
\newblock Random planar metrics.
\newblock In {\em Proceedings of the {I}nternational {C}ongress of
  {M}athematicians. {V}olume {IV}}, pages 2177--2187. Hindustan Book Agency,
  New Delhi, 2010.

\bibitem{bm2017disk}
J.~Bettinelli and G.~Miermont.
\newblock Compact {B}rownian surfaces {I}: {B}rownian disks.
\newblock {\em Probab. Theory Related Fields}, 167(3-4):555--614, 2017.

\bibitem{bg2008geodesics}
J.~Bouttier and E.~Guitter.
\newblock Statistics in geodesics in large quadrangulations.
\newblock {\em J. Phys. A}, 41(14):145001, 30, 2008.

\bibitem{cl2014bp}
N.~Curien and J.-F. Le~Gall.
\newblock The {B}rownian plane.
\newblock {\em J. Theoret. Probab.}, 27(4):1249--1291, 2014.

\bibitem{dddf2019tightness}
J.~{Ding}, J.~{Dub{\'e}dat}, A.~{Dunlap}, and H.~{Falconet}.
\newblock {Tightness of Liouville first passage percolation for $\gamma \in
  (0,2)$}.
\newblock {\em arXiv e-prints}, page arXiv:1904.08021, Apr 2019.

\bibitem{dk1998spectra}
B.~Duplantier and I.~Kostov.
\newblock Conformal spectra of polymers on a random surface.
\newblock {\em Phys. Rev. Lett.}, 61(13):1433--1437, 1988.

\bibitem{dk1990}
B.~Duplantier and I.~K. Kostov.
\newblock Geometrical critical phenomena on a random surface of arbitrary
  genus.
\newblock {\em Nuclear Phys. B}, 340(2-3):491--541, 1990.

\bibitem{dms2014mating}
B.~{Duplantier}, J.~{Miller}, and S.~{Sheffield}.
\newblock {Liouville quantum gravity as a mating of trees}.
\newblock {\em ArXiv e-prints}, Sept. 2014.

\bibitem{drsv2014critical_deriv}
B.~Duplantier, R.~Rhodes, S.~Sheffield, and V.~Vargas.
\newblock Critical {G}aussian multiplicative chaos: convergence of the
  derivative martingale.
\newblock {\em Ann. Probab.}, 42(5):1769--1808, 2014.

\bibitem{drsv2014critical_kpz}
B.~Duplantier, R.~Rhodes, S.~Sheffield, and V.~Vargas.
\newblock Renormalization of critical {G}aussian multiplicative chaos and {KPZ}
  relation.
\newblock {\em Comm. Math. Phys.}, 330(1):283--330, 2014.

\bibitem{ds2011kpz}
B.~Duplantier and S.~Sheffield.
\newblock Liouville quantum gravity and {KPZ}.
\newblock {\em Invent. Math.}, 185(2):333--393, 2011.

\bibitem{gm2016saw}
E.~{Gwynne} and J.~{Miller}.
\newblock {Convergence of the self-avoiding walk on random quadrangulations to
  SLE$_{8/3}$ on $\sqrt{8/3}$-Liouville quantum gravity}.
\newblock {\em ArXiv e-prints}, Aug. 2016.
\newblock To appear in Ann. Sci. \'{E}c. Norm. Sup\'{e}r.

\bibitem{gm2017halfplane}
E.~Gwynne and J.~Miller.
\newblock Scaling limit of the uniform infinite half-plane quadrangulation in
  the {G}romov-{H}ausdorff-{P}rokhorov-uniform topology.
\newblock {\em Electron. J. Probab.}, 22:Paper No. 84, 47, 2017.

\bibitem{gm2019confluence}
E.~{Gwynne} and J.~{Miller}.
\newblock {Confluence of geodesics in Liouville quantum gravity for $\gamma \in
  (0,2)$}.
\newblock {\em arXiv e-prints}, page arXiv:1905.00381, May 2019.
\newblock To appear in Annals of Probability.

\bibitem{gm2019conf}
E.~{Gwynne} and J.~{Miller}.
\newblock {Conformal covariance of the Liouville quantum gravity metric for
  $\gamma \in (0,2)$}.
\newblock {\em arXiv e-prints}, page arXiv:1905.00384, May 2019.

\bibitem{gm2017disk}
E.~Gwynne and J.~Miller.
\newblock Convergence of the free {B}oltzmann quadrangulation with simple
  boundary to the {B}rownian disk.
\newblock {\em Ann. Inst. Henri Poincar\'{e} Probab. Stat.}, 55(1):551--589,
  2019.

\bibitem{gm2019exunique}
E.~{Gwynne} and J.~{Miller}.
\newblock {Existence and uniqueness of the Liouville quantum gravity metric for
  $\gamma \in (0,2)$}.
\newblock {\em arXiv e-prints}, page arXiv:1905.00383, May 2019.

\bibitem{gm2019localmetrics}
E.~{Gwynne} and J.~{Miller}.
\newblock {Local metrics of the Gaussian free field}.
\newblock {\em arXiv e-prints}, page arXiv:1905.00379, May 2019.

\bibitem{gms2018poissonvoronoi}
E.~{Gwynne}, J.~{Miller}, and S.~{Sheffield}.
\newblock {The Tutte embedding of the Poisson-Voronoi tessellation of the
  Brownian disk converges to $\sqrt{8/3}$-Liouville quantum gravity}.
\newblock {\em ArXiv e-prints}, Sept. 2018.

\bibitem{hk1971quantum}
R.~H{\o}egh-Krohn.
\newblock A general class of quantum fields without cut-offs in two space-time
  dimensions.
\newblock {\em Comm. Math. Phys.}, 21:244--255, 1971.

\bibitem{lv2012multifractal}
F.~Johansson~Viklund and G.~F. Lawler.
\newblock Almost sure multifractal spectrum for the tip of an {SLE} curve.
\newblock {\em Acta Math.}, 209(2):265--322, 2012.

\bibitem{js2000remove}
P.~W. Jones and S.~K. Smirnov.
\newblock Removability theorems for {S}obolev functions and quasiconformal
  maps.
\newblock {\em Ark. Mat.}, 38(2):263--279, 2000.

\bibitem{k1985gmc}
J.-P. Kahane.
\newblock Sur le chaos multiplicatif.
\newblock {\em Ann. Sci. Math. Qu\'ebec}, 9(2):105--150, 1985.

\bibitem{law2005conformally}
G.~F. Lawler.
\newblock {\em Conformally invariant processes in the plane}, volume 114 of
  {\em Mathematical Surveys and Monographs}.
\newblock American Mathematical Society, Providence, RI, 2005.

\bibitem{lg2013uniqueness}
J.-F. Le~Gall.
\newblock Uniqueness and universality of the {B}rownian map.
\newblock {\em Ann. Probab.}, 41(4):2880--2960, 2013.

\bibitem{mmq2018welding}
O.~{McEnteggart}, J.~{Miller}, and W.~{Qian}.
\newblock {Uniqueness of the welding problem for SLE and Liouville quantum
  gravity}.
\newblock {\em ArXiv e-prints}, page arXiv:1809.02092, Sept. 2018.
\newblock To appear in J. Inst. Math. Jussieu.

\bibitem{m2013brownian}
G.~Miermont.
\newblock The {B}rownian map is the scaling limit of uniform random plane
  quadrangulations.
\newblock {\em Acta Math.}, 210(2):319--401, 2013.

\bibitem{ms2015axiomatic}
J.~{Miller} and S.~{Sheffield}.
\newblock {An axiomatic characterization of the Brownian map}.
\newblock {\em ArXiv e-prints}, June 2015.

\bibitem{ms2015lqg_tbm1}
J.~{Miller} and S.~{Sheffield}.
\newblock {Liouville quantum gravity and the Brownian map I: The QLE(8/3,0)
  metric}.
\newblock {\em ArXiv e-prints}, July 2015.
\newblock To appear in {I}nventiones.

\bibitem{ms2016lqg_tbm2}
J.~{Miller} and S.~{Sheffield}.
\newblock {Liouville quantum gravity and the Brownian map II: geodesics and
  continuity of the embedding}.
\newblock {\em ArXiv e-prints}, May 2016.

\bibitem{ms2016lqg_tbm3}
J.~{Miller} and S.~{Sheffield}.
\newblock {Liouville quantum gravity and the Brownian map III: the conformal
  structure is determined}.
\newblock {\em ArXiv e-prints}, Aug. 2016.

\bibitem{rs2005basic}
S.~Rohde and O.~Schramm.
\newblock Basic properties of {SLE}.
\newblock {\em Ann. of Math. (2)}, 161(2):883--924, 2005.

\bibitem{s2000scaling}
O.~Schramm.
\newblock Scaling limits of loop-erased random walks and uniform spanning
  trees.
\newblock {\em Israel J. Math.}, 118:221--288, 2000.

\bibitem{ss2013continuum_contour}
O.~Schramm and S.~Sheffield.
\newblock A contour line of the continuum {G}aussian free field.
\newblock {\em Probab. Theory Related Fields}, 157(1-2):47--80, 2013.

\bibitem{sw2005coordinate}
O.~Schramm and D.~B. Wilson.
\newblock S{LE} coordinate changes.
\newblock {\em New York J. Math.}, 11:659--669, 2005.

\bibitem{s2007gff}
S.~Sheffield.
\newblock Gaussian free fields for mathematicians.
\newblock {\em Probab. Theory Related Fields}, 139(3-4):521--541, 2007.

\bibitem{s2016zipper}
S.~Sheffield.
\newblock Conformal weldings of random surfaces: {SLE} and the quantum gravity
  zipper.
\newblock {\em Ann. Probab.}, 44(5):3474--3545, 2016.

\bibitem{wu2018arms}
H.~Wu.
\newblock Alternating arm exponents for the critical planar {I}sing model.
\newblock {\em Ann. Probab.}, 46(5):2863--2907, 2018.

\end{thebibliography}

\end{document}